\documentclass{cedram-aif}

\newcommand{\T}{{\mathbb T}}

\usepackage{color}

\newcommand{\dbar}{\overline{\partial}}
\newcommand{\ddbar}{\partial \overline{\partial}}

\newcommand{\kcalomega}{\kcal_{\omega}}
\usepackage{epsfig}
\usepackage{amscd}
\usepackage{graphics}
\usepackage{color}
\newcommand{\ncal}{{\mathcal N}}

\newcommand{\half}{\frac{1}{2}}
\newcommand{\kahler}{K\"ahler }
\newcommand{\RP}{{\mathbb R} {\mathbb P}}

\newcommand{\C}{\mathbb{C}}

\newcommand{\R}{\mathbb{R}}

\newcommand{\Z}{\mathbb{Z}}

\newcommand{\acal}{\mathcal{A}}

\newcommand{\ccal}{\mathcal{C}}

\newcommand{\gcal}{\mathcal{G}}
\newcommand{\hcal}{\mathcal{H}}

\newcommand{\kcal}{\mathcal{K}}

\newcommand{\pcal}{\mathcal{P}}

\newcommand{\ocal}{\mathcal{O}}

\newcommand{\zcal}{\mathcal{Z}}

\usepackage{mathrsfs, amsmath,amssymb,amsfonts, amsthm}
\newtheorem{conjecture}{Conjecture}[section]
\newtheorem{condition}{Condition}[section]
\numberwithin{equation}{section}

\begin{document}

\title[Boundedness of the number of nodal domains]{Boundedness  of the number of nodal domains for eigenfunctions of generic Kaluza-Klein  3-folds}
\alttitle{Limite du nombre de domaines nodales des  fonctions propres de vari\'et\'ees Kaluza-Klein generiques en dimension 3}
\author{\firstname{Junehyuk} \lastname{Jung}}

\address{Department of Mathematics, Texas A\&M University, College Station, TX 77843-3368 USA}

\email{junehyuk@math.tamu.edu}

\author{\firstname{Steve} \lastname{Zelditch}}
\address{Department of Mathematics, Northwestern  University, Evanston, IL 60208, USA}

\email{zelditch@math.northwestern.edu}

\thanks{Research partially supported by NSF grant  DMS-1541126. The first author is partially supported by Sloan Research Fellowship and by NSF grant DMS-1900993. The first author thanks Eviatar Procaccia for an enlightening discussion.}

\altkeywords{fonction propre du Laplacien, fibr\'e principale,  metrique de Kaluza-Klein, domaine nodale}

\begin{abstract}
This article concerns the number of nodal domains of eigenfunctions of
the Laplacian on special Riemannian  $3$-manifolds, namely nontrivial principal $S^1$ bundles
$P \to X$  over Riemann surfaces equipped with certain $S^1$ invariant metrics, the Kaluza-Klein metrics. We prove for generic Kaluza-Klein metrics that any Laplacian eigenfunction has exactly two nodal domains unless it is invariant under the $S^1$ action.

We also construct an explicit orthonormal eigenbasis on the flat $3$-torus $\mathbb{T}^3$ for which every non-constant eigenfunction has two nodal domains.
\end{abstract}
\begin{altabstract}
Cet article concerne le nombre de domaines nodales des fonctions propres du Laplacien sur des vari\'et\'ees Riemanniennes Kaluza-Klein en dimension trois,
\`a savoir des vari\'et\'ees qui sont fibr\'es $S^1$-principales  $P \to X$ sur des surfaces de Riemann \'equip\'ees avec une m\'etrique $S^1$-invariante de
type  Kaluza-Klein. Pour  metriques g\'en\'eriques de ce type, on preuve que chaque fonction propre poss\`ede exactement deux domains
nodales, sauf s'il est invariant par l'action de $S^1$.

On construit aussi une base orthonormale de fonctions propres explicite du tore plat ${\mathbb T}^3$ pour que chaque fonction propre non-constant poss\`ede
exactement deux domaines nodales.
\end{altabstract}
\maketitle


\section{Introduction}

This article is concerned with the number of nodal domains of eigenfunctions
of the Laplacian on certain $3$-dimensional compact smooth Riemannian  manifolds $(P, G)$.
The manifolds are $S^1 = SO(2)$ bundles $\pi:  P \to X$ over a Riemannian
surface $(X,g)$, and $G$ is assumed to be a Kaluza-Klein metric {\it adapted} to $\pi$, i.e.,  $G$ is invariant under the free $S^1$ action
on $P$ and
there exists a splitting $TP = H(P) \oplus V(P)$ of $TP$ so that   $d \pi: H_p(P) \to T_{\pi(p)} X$ is isometric and so the fibers are geodesics. Thus, $\pi: P \to X$
is a special kind of Riemannian submersion with totally geodesic fibers in the sense of \cite{BBB}
 (see Definition
\ref{SASAKIDEF} and Definition \ref{ADAPTEDDEF}).
The $S^1$ action commutes with the
 the Laplacian  $\Delta_G$ of the Kaluza-Klein metric $G$ and one may separate
 variables to obtain an orthonormal basis of  joint eigenfunctions $\phi_{m,j}$, \begin{equation} \label{DELTAG}
 \Delta_G \phi_{m,j} = - \lambda_{m,j} \phi_{m,j}, ~
 \frac{\partial}{\partial \theta} \phi_{m,j} = i m \phi_{m,j}. \end{equation}

Our focus is on the nodal sets of the real or imaginary parts of
\begin{equation}\label{uv}
\phi_{m,j}= u_{m,j} + i v_{m,j}
\end{equation}
and on particularly on the number of their nodal domains.
Since $\Delta_G$ is a real operator, the real  and imaginary parts \eqref{uv} satisfied the modified eigenvalue system,
\begin{equation*}
\left\{ \begin{array}{l}   \Delta_G  u_{m,j} =  -\lambda_{m,j}  u_{m,j}, \\ \\
   \Delta_G  v_{m,j} = -\lambda_{m,j} v_{m,j}, \\ \\
   \frac{\partial}{\partial \theta} u_j = m v_j, \;\; \frac{\partial}{\partial \theta}  v_j = -m u_j.
   \end{array} \right.
\end{equation*}

   Our main result (Theorem \ref{MAIN}) is that when $0$ is a regular value of $\phi_{m,j}$ for all $(m,j)$, then for $m \not=0$,  the nodal sets of  $u_{m,j}$,
   resp.   $v_{m,j}$, are connected and there exist exactly $2$ corresponding nodal domains. The case $m = 0$ is special because
   $\phi_{0, j}$ is then real valued and is  pullback from the base $X$;
   in this  case,  the number of connected components of the nodal set (and the number of nodal domains)  is the same
   as for the corresponding eigenfunction on $X$.
   Theorem \ref{GENERIC} shows  that it is a generic property of Kaluza-Klein metrics
   on $S^1$ bundles over Riemann surfaces that $0$ is indeed a regular value of $\phi_{m,j}$  for all $(m,j)$. The precise statement requires
   a discussion of the geometric data underlying a Kaluza-Klein metric and how we allow it to vary when defining `genericity'. An introductory
   discussion of the Kaluza-Klein metrics of this article  is given in Section \ref{KKINTRO} and a more detailed discussion
   is given in Section \ref{PERTURBATIONSECT} (see Theorem \ref{GENERIC1} and Lemma \ref{simplesimple}).

\subsection{\label{KKINTRO} Adapted Kaluza-Klein metrics}

We now define Kaluza-Klein metrics on a three-dimensional manifold $P$ which is an $S^1$ bundle over a (usually) compact
Riemannian surface $X$.  In our main results, $P$ is the unit co-circle bundle of an ample complex holomorphic line bundle $L \to X$.
Thus, $P = S^3$ or $P = U(S^2)$ (the unit tangent bundle) when $X = S^2$, and $P = S^* X$ (the unit conormal bundle) when the genus $g \geq 2$. When $g =1$, $P$ can be the
unit circle bundle of the theta line bundle over $\T^2$, or it may  be the trivial circle bundle $S^* \T^2$ (see Section \ref{FLATTORISECT}). Other examples where $X$ has constant curvature are discussed in  Section \ref{SCCSECT}.

A Kaluza-Klein metric is determined by the following data:
\begin{itemize}
\item[(i)]  A surface $X$ equipped with a Riemannian metric $g$ and a complex
structure $J$,

\item[(ii)]  A  nontrivial complex holomorphic  line bundle $L \to X$ over a surface,

\item[(iii)]  A Hermitian metric $h$ on $L$,

\item[(iv)]  A complex structure $J_L$ on $L$,

\item[(v)]  An $h$-compatible  connection $\nabla$ on $L$.

\end{itemize}
In this article, we fix $J$, $L$ and $J_L$ and only vary the
data $(g, h, \nabla)$.  The unitary frame bundle for the Hermitian metric $h$ is defined by  $$P_h = \{(z, \lambda) \in L^*: h_z^*(\lambda) = 1\}.$$
The connection $\nabla$ induces a connection $1$-form on $P_h$ and a splitting $T P_h = H(P_h) \oplus V(P_h)$ into horizontal and vertical
spaces; see Section \ref{GEOMSECT} for background.

\begin{defi} \label{SASAKIDEF}The Kaluza-Klein metric on
 $P_h$ is the  $U(1)$-invariant metric $G$ such that the horizontal
space $H_{p} : = \ker d\pi$ is isometric to $T_{\pi(p)} X$, so that $V = \R \frac{\partial}{\partial \theta}$ is
orthogonal to $H$ and is invariant under the natural $S^1$ action and so that the orbits of the $S^1$ action  are unit speed vertical geodesics. \end{defi}

The data $(g, h, \nabla)$ determines a horizontal
Laplacian $\Delta_H$,  a vertical Laplacian $\frac{\partial^2}{\partial \theta^2}$, and their sum, the Kaluza-Klein Laplacian
\eqref{DELTAG},
\begin{equation} \label{DELTAG2} \Delta_G = \Delta_H + \frac{\partial^2}{\partial \theta^2}.  \end{equation}
As is well-known, sections $s$ of powers $L^m$ of a complex  line bundle $L$ lift to equivariant functions $\hat{s}: L^* \to \C$
on the dual line bundle by the formula $\hat{s}(z, \lambda) = \lambda(s(z))$ where a point of $L^*$ is denoted by $\lambda \in L_z^*$.
 Under this identification, the  horizontal Laplacian is equivalent to the  Bochner Laplacians $\nabla_m^* \nabla_m$ on sections of $L^m$.
Thus, equivariant eigenfunctions of $\Delta_G$ of weight $m$ on $M$ are lifts of eigensections of $\nabla_m^*\nabla_m$. See Section
\ref{BOCHNERSECT} and   Lemma \ref{BOCHNERLEM} for details.  In proving genericity
theorems it is easier to work downstairs on $X$. But the nodal results pertain to the equivariant eigenfunctions on $P_h$.

\begin{rema}
Not all $S^1$ invariant metrics on  $SX$ are adapted Kaluza-Klein metrics (see Section \ref{ADAPTEDSECT}). Many of the techniques
of this paper extend with no major modifications to general $S^1$-invariant metrics on principal $S^1$ bundles over manifolds of all dimensions.
For simplicity of exposition we restrict to dimension $3$. \end{rema}

 \subsection{Nodal sets}
We thus have two versions of the eigenfunctions of the Kaluza-Klein $\Delta_G$,
first as scalar complex valued equivariant eigenfunctions on $P_h$ and second
as complex eigensections on $X$. In each version we have a nodal set,
and we use the base nodal set on $X$ to analyse the nodal set on $P_h$.

We denote the eigensection corresponding to $\phi_{m,j}$ as $f_{m,j} e_L^m$ in a local holomorphic frame.
We mainly consider $L = K_X$ and then we write the section as $f_{m,j}(z) (dz)^m$.
Let
\begin{equation*}
\Re f_{m,j} = a_{m,j}(z), \;\; \Im f_{m,j} = b_{m,j}(z).
\end{equation*}
Then,
\[
f_{m,j}(z) e^{-i m\theta} = (a_{m,j}(z) + i  b_{m,j}(z))(\cos m \theta - i \sin m \theta),
\]
so that with $\phi_{m, j} = u_{m,j} + i v_{m,j}$,
 \begin{equation} \label{ubab} \left\{ \begin{array}{l}
 u_{m,j} =  a_{m,j} \cos m \theta  +  b_{m,j} \sin m\theta,\\
v_{m,j} = b_{m,j} \cos m \theta  -a_{m,j} \sin m \theta.
\end{array} \right. \end{equation}
See Section \ref{LIFTSECT} for more details.

We denote by $\zcal_{f_{m,j}}$ the  zero set of the eigensection $f_{m,j} e_L^m$ on $X$:
\begin{equation*}
\zcal_{f_{m,j}} = \{z \in X: f_{m,j}(z) = 0\}.
\end{equation*}
It is easy to see that the zero set $\zcal_{\phi_{m,j}}$ of $\phi_{m,j}$ is the inverse image of $\zcal_{f_{m,j}}$
under the natural projection $\pi$:
\begin{equation*}
\zcal_{\phi_{m,j}}  = \pi^{-1} \zcal_{f_{m,j}}.
\end{equation*}
Usually we study the nodal sets of the real and imaginary parts of the lift, not to be confused with the lifts of the real
and imaginary parts of the local expression $f_{m,j}$ of the section (since the frame $e_L^m$ must also be taken into account).
 In general, it  is not obvious whether or not  the zero set of $f_{m,j}$ is discrete in $X$.

 We  denote the nodal sets of the real, resp. imaginary parts, of the lift by  $$\ncal_{ \Re \phi_{m,j}}  \{p \in P_h: \Re \phi_{m,j}(p) = 0\}, ~\mathrm{ resp. }~  \ncal_{\Im \phi_{m,j}} = \{p \in P_h: \Im \phi_{m,j}(p) = 0\}.$$
The analysis is the same for real and imaginary parts and we generally work with the imaginary part, following the tradition for quadratic differentials.

Perhaps the most familiar setting for such Kaluza-Klein metrics and lifts
of $m$-differentials to equivariant sections is that of hyperbolic surfaces
of finite area. The reader familiar with Maass forms and operators may want
to compare Kaluza-Klein notions  with those of $SL(2, \R)$-theory
in Section \ref{HYPERBOLICSECT}.

\subsection{Statement of results}
 Let $(X, J, g)$ be a Riemannian surface with  complex structure $J$. Let $(L, h)$ be a Hermitian
 holomorphic line bundle over $X$, and let
 $P_h \subset L^* \to X$ be the principal $S^1$ bundle associated to $h$. Let $\nabla_h$ be an $h$ compatible connection
 on $L$, and let   $G = G(g, h, \nabla_h)$ be the associated Kaluza-Klein metric on $P_h$.

 As mentioned above, weight $0$ (invariant) eigenfunctions are special because they are real-valued (once they are multiplied by a suitable constant). We therefore separate the
case $m = 0$ from the remainder of the discussion, and state  the obvious (but interesting)
\begin{prop} \label{m=0} For $m = 0$,   invariant eigenfunctions ($m = 0$) of $\Delta_G$ are lifts $\pi^* \psi_{j}$ of eigenfunctions $\psi_j$ of $\Delta_g$ on the base $X$, and the nodal set of
 $\pi^* \psi_{j} $ is the inverse image under $\pi$ of the nodal set
 of $\psi_{j}$.  The number of nodal domains of $\pi^* \psi_{j}$ equals
 the number of nodal domains of $ \psi_{j}$.

\end{prop}

Indeed,
    their  nodal sets are inverse images of nodal sets on the base. Hence the number of nodal domains of `invariant'  Kaluza-Klein eigenfunctions is the number for the corresponding eigenfunction on the base.
Henceforth we always assume $m \not= 0$.

  To prepare for our main result when $m \not=0$, we first state a result on generic properties of equivariant Kaluza-Klein eigenfunctions. By  `generic' properties
 of Kaluza-Klein metrics, we mean properties of residual sets in a suitable $C^k$ space of the data $(g, h, \nabla)$, often when only one component is
 varied and the others are fixed.   The generic properties of concern in this article  hold for  many choices of
 Banach spaces of data, which could be suitable  $C^k$ spaces or a Sobolev spaces  $H^s$. Our basic reference for genericity properties of eigenvalues/eigenfunctions in \cite{U}, and the reader is referred there for background. Somewhat surprisingly, generic properties of eigensections of Bochner-Kodaira operators on complex line bundles do not seem to have been studied before.

   The most general
 genericity results are stated in Theorem \ref{GENERIC1} in Section \ref{PERTURBATIONSECT}. In these results,  we define `admissible data' such as  $(g, h, \nabla)$, and prove the main genericity properties as different components of  the data  is varied. Since the general result
 requires some definitions from Sections \ref{GEOMSECT}-\ref{BOCHNERSECT}, we only state the most elementary result here in the case where $L = L = T^*X$ is the canonical bundle,
 and where only the metric $g$ on $X$ is varied. But we also require
 that the Hermitian metric $h$ on $K$ is induced by $g$ and that the connection $\nabla_h$ is the Chern (or equivalently, Riemannian) connection. Of course, $K$ s not ample when $g = 0,1$ but we are considering eigensections of Bochner-Kodaira Laplacians, not holomorphic sections, so ample-ness is not particularly relevant. Moreover, the proofs
 also work of $K^{-1}$, so that one could replace $T^* S^2$ by the ample line bundle $T S^2$ or $\ocal(1) \to {\mathbb \C} {\mathbb P}^1$ and obtain similar results.

\begin{theo}\label{GENERIC}  Let $(X, J)$ be a Riemann surface, and let  $L = K^m$. We consider Riemannian metrics  $g$ in the conformal class  associated to $J$. We assume that $h$ is the Hermitian metric induced by $g$ and
that  $\nabla$ is the Levi-Civita connection. Then, for  generic  metrics $g$ in the class of $J$ on $X$,
\begin{enumerate}
\item the spectrum of
each Bochner Laplacians $\nabla_{g,h}^* \nabla$ on $C^k(X, L^m)$ is simple (i.e. of multiplicity $1$). Thus, the multiplicity of the eigenvalue $\lambda=\lambda_{m,j}$ of $\Delta_G$ is $1$ if $m=0$, and $2$ if $m\neq 0$.

\item Every eigenfunction is a joint eigenfunction of $\Delta_G$ and $\frac{\partial^2}{\partial \theta^2}$.

  \item  all of the eigensections $f_{m, j}$ have
isolated zeros and zero is a regular value. In particular,  $\zcal_{f_{j ,m}} $ is a finite set of points;

\item
 If we lift sections to equivariant eigenfunctions $\phi$,
then $\Re \phi$ and $\Im \phi$ have zero as a regular value.

\end{enumerate}
 \end{theo}

 An important consequence of  (1)-(2) of Theorem \ref{GENERIC} is that, for generic
 data, all real Kaluza-Klein eigenfunctions are real/imaginary parts of  equivariant eigenfunctions. Therefore, the results
 we prove for equivariant  eigenfunctions  hold for all possible eigenfunctions.
Theorem \ref{GENERIC1} states similar results for three other types of variations of the data.

We can now state our main result. The first statement repeats (1)-(2)
of Theorem \ref{GENERIC} for the sake of clarity.

 \begin{theo}\label{MAIN} Suppose that the data $(g, h, \nabla)$ of
 the Kaluza-Klein metric satisfies the generic properties of Theorem \ref{GENERIC}. Then,

 \begin{enumerate}

\item The eigenspace of $\Delta_G$ corresponding to $\lambda=\lambda_{m,j}=\lambda_{-m,j}$ is spanned by $ \phi_{m,j}$ and  $\phi_{-m,j}= \overline{\phi_{m,j}}$. In particular, any real eigenfunction with the eigenvalue $\lambda_{m,j}$ is a constant multiple of $T_\theta \left(\Re \phi_{m,j}\right) $, where $T_\theta$ is the $S^1$ action on $P$ parameterized by $\theta$.

 \item For $m \neq 0$, the nodal sets of $\Re \phi_{m,j}$  are connected.

   \item For $m \neq 0$, the number of  nodal domains of $\Re \phi_{m,j}$  is $2$.
 \end{enumerate}

 \end{theo}

\noindent Note from Weyl law that $\#\{\lambda_{j,0}<\Lambda\} \sim \Lambda$ and that $\#\{\lambda_{m,j}<\Lambda\} \sim \Lambda^{3/2}$. Therefore as an immediate consequence of Theorem \ref{MAIN}, we have the following:
\begin{coro}
Let $P \to X$ be a non-trivial principal $S^1$ bundle with a generic Kaluza-Klein metric. For any given orthonormal eigenbasis, almost all (i.e., along a subsequence of density one) eigenfunctions have exactly two nodal domains.
\end{coro}

The density one subsequence is of course the one with $m \not=0$.
 Theorem \ref{MAIN} furnishes the first example of Riemannian manifolds
of dimension $> 2$ for which the number of nodal domains and connected
components of the nodal set have been counted precisely. The results
for $m \not=0$ may seem rather surprising, since in dimension $2$ the only
known sequences of eigenfunctions with a bounded number of nodal domains are those constructed in an ingenious way by H. Lewy on the standard $S^2$ \cite{Lew77} and those of Stern on a flat torus \cite{CH2,CH}. In those cases,
the separation-of-variables eigenfunctions have connected nodal sets but the complement of the nodal set has many components, i.e., nodal domains, saturating the Courant bound that the number
of nodal domains of the $j$th eigenfunction (in order of increasing eigenvalue) is $j$. In the Kaluza-Klein case, all eigenfunctions for generic
Kaluza-Klein metrics are separation-of-variables eigenfunctions and have
connected nodal sets. But the connectivity is of a different kind than
in dimension two and by a simple argument it induces connectivity of nodal
domains.

\begin{rema}
When $P \to X$ is trivial, and $P\cong S^1 \times X$ is endowed with the product metric, we have $\phi_{m,j}=\psi_j e^{im\theta}$ where $\psi_j$ is an eigenfunction of $\Delta_g$ on the base $X$. Hence $\Re \phi_{m,j} = \psi_j\cos m\theta$ has many nodal domains, and the last statement in Theorem \ref{MAIN} fails. Hence, the `generic' set of metrics is not the
full set of metrics. See Section \ref{FLATTORISECT} for flat tori for a product setting where the nodal results above do hold.

\end{rema}

\subsection{Outline of the proof}

 The nodal set of the lift real and imaginary parts of the lift $\phi_{m,j}$ of $f_{m,j} e_L^m$  is very different over nodal versus non-nodal points of $f_{m,j}$.
Let \begin{equation} \label{SIGMADEF} \Sigma: = \pi^{-1}\zcal_{f_{m,j}} \end{equation}  be the inverse image of the base nodal points. It is a union of fibers and is a finite union  of fibers
if and only if $f_{m,j}$ has a finite number of zeros. We refer to $\Sigma $ as the `singular fibers' or singular set.
We denote by $X \backslash \zcal_{f_{m,j}}$ the punctured Riemann surface in which the
zero set of $f_{m,j}$ is deleted.  A key statement in the nodal analysis is the following:

\begin{prop} \label{COVER} For $m \not= 0$, the  maps
$$\pi: \ncal_{\Re \phi_{m,j}} \backslash \Sigma \to X \backslash \zcal_{f_{m,j}},\;\;  \ncal_{\Im \phi_{m,j}} \backslash \Sigma \to X \backslash \zcal_{f_{m,j}}\;\; $$
is an $m$-fold covering space. \end{prop}

It follows that the topology of the nodal set is entirely determined by the combinatorics of gluing
the sheets along the singular fibers. In fact, the gluing is rather simple and easily yields the following

\begin{theo}\label{NODALCON} For all $m\neq 0$, the nodal set  $\ncal_{\Re \phi_{m,j}}$
is connected. \end{theo}

To count nodal domains, we need to make the assumption that there are just a finite number of zeros of $f_{m,j}$ and that at least one of them is regular.

When the zero set is transverse to the zero section, then the sum of the indices of the zeros
is the first Chern class of $K^m$, and in particular is non-empty when the genus of $X$ is $\not= 0$, i.e., when $X$ is not a torus.

For metrics satisfying Theorem \ref{GENERIC} we prove Theorem \ref{MAIN} by using Proposition \ref{COVER} together with
some geometric observations on how the sheets fit together at the singular fibers. This is done using a Bers type local analysis
of the eigensections (Section \ref{LOCALSECT}) and some geometric/combinatorial arguments in Section \ref{NODALDOMAINSECT1}.

To put the nodal results into context,  it is proved in varying degrees of generality in   \cite{GRS13,JZ13,JZ14, Z17, GRS15,JS,JY,MM}
that in dimension $2$, the number of nodal domains of an orthonormal
basis $\{u_j\}$ of Laplace  eigenfunctions on certain surfaces with ergodic geodesic flow tends to infinity with the eigenvalue along almost the entire sequence
of eigenvalues. By the first item of Theorem \ref{MAIN}, the same is true
for their lifts to the unit  tangent bundle $SX$ as  invariant eigenfunctions of the Kaluza-Klein metric. But for higher weight eigenfunctions, the situation is
virtually the opposite and the number of nodal domains is bounded.
\begin{rema}
Note that the geodesic flow on $P$ with a Kaluza--Klein metric never is ergodic. To see this, observe that the hypersurface
\[
\{(x,v)\in SP ~:~ v\perp \partial_\theta \} \subset SP
\]
is invariant under the geodesic flow, and it divides $SP$ into two subsets of positive measure:
\[
\{(x,v)\in SP ~:~ \langle v, \partial_\theta \rangle_G >0 \} ~\text{ and }~ \{(x,v)\in SP ~:~ \langle v, \partial_\theta \rangle_G <0 \}.
\]
\end{rema}

\subsubsection{Surfaces of constant curvature}
Metrics $g$ of constant curvature with their associated Hermitian metrics
and connections on $K \to X$ are not generic. But they are of special interest, so we comment on what we are able to prove about them. Note
that the standard metric on $\T^3$ is Kaluza-Klein, as is the standard
metric on $S^3$ or $SO(3) = U(S^2)$. The standard metric on the unit
tangent bundle $PSL(2,\R)/\Gamma$ over a hyperbolic surface is Lorentzian, but if one changes the sign of the vertical Laplacian it is also Kaluza-Klein.

Perhaps surprisingly, the results of Theorem \ref{MAIN} are valid for some orthonormal bases of eigenfunctions  on
flat $3$-tori.
\begin{theo}\label{explicit}
On the flat $3$ torus $\mathbb{T}^3$, one can find an orthonormal eigenbasis for which all nonconstant eigenfunctions have two nodal domains.
\end{theo}

Next we turn to hyperbolic metrics on a surface $X$ of genus $g \geq 2$.
 Then
the total space $P_h = PSL(2, \R)/\Gamma$ and the equivariant eigenfunctions
of
the Kaluza-Klein Laplacian $\Delta$ are the same as joint eigenfunctions of the generator $W$ of $K$ and of the Casimir operator
$\Omega$.
When the weight $m$ is fixed, one may separate variables and obtain a Maass Laplacian $D_m$ on smooth sections of a complex line bundle $\pi: K^m \to X$, namely the bundle of $m$-differentials of
type $(dz)^m$, and are the usual weight $m$ automorphic Maass eigendifferentials $f_{m,j}(z)(dz)^m$,
$$D_m f_{m,j}(z) = s(1-s) f_{m,j}(z),$$
of  the Maass Laplacians
$$D_m =  y^2 (\frac{\partial^2}{\partial x^2} + \frac{\partial^2}{\partial y^2}) - 2 i m y\frac{\partial }{\partial x}.$$
Unfortunately, we are not able to verify that (any) Maass eigendifferentials
have $0$ as a regular value, i.e.  the  generic  conditions needed for Theorem \ref{MAIN}. Indeed, we do not know how to prove that (any) eigenfunctions $\phi_{0, j}$  of the hyperbolic Laplacian have a discrete set of critical points, much less that $0$ is a regular value of $d \phi_{0, j}$.

Regarding spherical harmonics on $S^3$ we will be brief, because
we study random equivariant spherical harmonics  in detail in a forthcoming article \cite{JZ2019}. The round metric is a Kaluza-Klein metric, but of course not
a generic one. In Section \ref{SPHERE} we show that the joint eigenfunctions $Y_N^{m_1, m_2}$ of $\Delta$ and of two commuting $S^1$ actions have very different nodal sets from the ones in Theorem \ref{MAIN}. On the other hand, in \cite{JZ2019} we show that `random' linear combinations of such
$Y_N^{m_1, m_2}$  with $m_1 $ fixed do satisfy the results of Theorem
\ref{MAIN} and the nodal sets of their real, resp. imaginary parts, have
just one nodal component. We also find their expected Euler characteristic.
These results were motivated by  the  numerical discovery of Barnett et. al. that  3D  random spherical harmonics on $S^3$ of fixed degree, the nodal set contains one `giant component' and many small components \cite{B18}.
We are fixing the weight $m_1$ and therefore do not work with general random spherical harmonics. But in the $N$ dimensional subspaces where
$m_1 $ is fixed with $|m_1| \leq N, 2 | N-m_1$,  the nodal sets of the real
and imaginary parts are connected and divide $S^3$ into just two components. For further discussion we refer to \cite{JZ2019}.

\section{\label{GEOMSECT} Geometric background}

In this section we discuss the geometric data that goes into the
construction of Kaluza-Klein metrics, which are defined in Definition
\ref{SASAKIDEF}. They are also the data needed to define Bochner
Laplacians $\nabla^* \nabla$ and Kaluza-Klein Laplacians $\Delta_G$.
We plan to vary the data and study perturbation theory of eigenvalues
and eigensections in Section \ref{PERTURBATIONSECT}.

\subsection{Riemannian metrics on $X$ and Hermitian metrics on $L$}

Let $(X, J, g)$ denote a Riemann surface with complex structure $J$ and
Riemannian metric $g$. We  write $g_{1 \bar{1}} = g(\frac{\partial}{\partial z}, \frac{\partial}{\partial z})$ and $g^{1 \bar{1}} = g^*(dz, dz)$, where $g^*$ is the dual metric. The complex structure gives a decomposition of $T^*X \otimes \C = T^{*(1,0)} \oplus T^{*(0,1)}$ into $(1,0)$ resp. $(0,1)$ parts. We denote the
area form of $g$ by $$dA_g = \omega = i g_{1 \bar{1}} dz \wedge d \bar{z}, $$  where the \kahler form $\omega$ is the
$(1,1)$ form defined by $g_J(X, Y) = \omega(J X, Y)$.  Locally there exists
a \kahler potential $\psi$ defined up to a constant by $dd^c \phi = \omega$.
Here $d^c = \frac{1}{2i}(\partial - \dbar)$. Then, $\omega_h = i \ddbar \log h = dd^c \phi = \Delta_g \phi L(dz)$, where $L(dz) = i dz \wedge d\bar{z}$.

\subsubsection{The space of isometry classes of metrics on surfaces}
It is well-known that the space of Riemannian metric tensors on a manifold $X$ splits as a product $\mathrm{Vol}(X) \times \mathrm{Met}_{\mu}(X)$ of volume forms times metrics with a fixed volume form.
Thus, one may separately  consider metrics with a fixed volume form and
conformal classes of metrics.

Choice of a complex structure $J$ on $X$ is equivalent to choice of a conformal class $\mathrm{Conf}(g_0) $
of metrics.
The moduli
space of conformal classes  is the same as the $(3g-3)$-dimensional  moduli space $\mathfrak M_g$ of complex structures
on $X$. In each conformal class, we may pick a background metric $g_0$ and represent other
metrics in the form $$ \mathrm{Conf}(g_0) = \{ e^{2 \sigma} g_0: \sigma \in C^{\infty}(X)\}.$$

 If we  fix  a complex structure $J$, then the Riemannian metrics in the corresponding conformal
 class are \kahler metrics, and may be parameterized  by their K\"ahler potentials $\phi$, where the
area form $\omega_{\phi}$  of the K\"ahler metric is related to that of the reference metric by
$$\omega_{\phi} = \omega_0 + i \ddbar \phi. $$ We let
$$\kcalomega : = \{\omega_{\phi}: = \omega_0 + i \ddbar \phi >0\} $$
which may be identified with an open set in $C^{\infty}(X)$.  The Liouville field and K\"ahler potential
are related by
\begin{equation*}  e^{2 \sigma} = 1 - \half \Delta_0 \phi, \end{equation*}
where $\Delta_0$ is the Laplacian of $g_0$.  The only difference in the two parameterizations
of conformal metrics is that the area of metrics in $\kcalomega$ is fixed while it may vary
in $\mathrm{Conf}(g_0)$. Thus $\mathrm{Conf}(g_0) \cong \kcalomega \times \R$.

In the case of Riemann surfaces, the area form is the symplectic
form associated to a \kahler metric.  Given a complex structure $J$, the  \kahler metric $g_J$ can be recovered from its area form $\omega$
by the formula $g_J(X, Y) = \omega(X, J Y)$. Hence isometry classes of \kahler metrics with a fixed
area form are parameterized by $\mathfrak M_g$.

\subsection{Complex line bundles $L \to X$, connections and curvature}

 If we fix a complex structure $J$ and holomorphic line bundle $L \to X$,
then a Hermitian metric $h$ on $L$ is determined by the length of a local holomorphic frame $e_L$ (i.e., a local holomorphic
 nonvanishing section) of $L$ over an open set $U\subset M$)
by $e^{-\psi} = \|e_L\|_h^2$, where
$\|e_L\|_h=h(e_L,e_L)^{1/2}$ denotes the $h$-norm of $e_L$.

\subsubsection{Connections}
 In the real setting, a connection on a vector bundle $E$ defines a covariant derivative
$$\nabla: C^{\infty}(X, E) \to C^{\infty}(X, E \otimes T^*X). $$

In our complex setting, we assume $L$ is a holomorphic Hermitian  line bundle, i.e., we equip $L$ with a
  complex structure $J_L$, a connection $\nabla$,  and a  Hermitian metric $h$.  In a local frame $e_L$ it is defined by $\nabla e_L = \alpha \otimes e_L$. We consider  several types of compatibility conditions between this data:

  \begin{itemize}

\item An $h$-connection $\nabla_h$ is one compatible with $h$. In a unitary
frame, the connection $1$-form is $i \R$ valued and is denoted by $i \alpha$.  We denote the space of $h$-compatible
connections by $\acal_h$.\bigskip

 \item Or a $J_L$-compatible connection. In a holomorphic frame $e_L$ the connection $1$-form $\alpha$
 is of type $(1,0)$. We denote the space of $J_L$-compatible connections by $\acal_{\C}$. \bigskip

 \item The unique Chern connection $\alpha_{J_L, h}$ which is compatible
 with both $J_L, h$.

 \end{itemize}

 \bigskip

This data induces:
\begin{itemize}

\item The Hermitian metric $h$ induces the principal bundle of $h$-unitary frames $P_h = \{(z, \lambda) \in L^*:
|\lambda|_h =1\}$. \bigskip

\item An $h$-compatible   connection $\nabla \in \acal_h$ induces  a real  $1$-form $\alpha$ on $P_h$.
 \bigskip

\item Connections  $\acal_{\C}$ determine complex-valued $1$-forms on $L^*$.

\end{itemize}

 The connection
1-form in the frame $e_L$ is given by $$\nabla e_L = \alpha \otimes e_L.$$
We denote the $(1,0)$ resp. $(0,1)$ parts of $\nabla $ by $\nabla^{1,0}, $ resp. $\nabla^{0,1}$.

 Suppose that $\nabla \in \acal_{\C}$.  Then
if $s = f e$ with $e$ a local holomorphic frame,
$$\left\{ \begin{array}{l} \nabla^{(1,0)} (fe) = (\partial f + \alpha f) \otimes e, \\ \\
\nabla^{(0,1)} (fe) = \dbar f \otimes e. \end{array} \right.$$

The  holomorphic line bundle $L$ also has a natural Cauchy-Riemann operator,
$$\dbar_L: C^{\infty}(M, L) \to C^{\infty}_1(M, L). $$
In a local holomorphic frame $e$, we write a smooth  section $s =
fe$ and then
$$\dbar_L s = \dbar f \otimes e. $$
It is well-defined since if $e'$ is another holomorphic frame and
$e = g e'$, then $s = f g e'$ and $\dbar_L s = \dbar f \otimes g
e' =\dbar f \otimes e$.

 The Chern
connection $\nabla$ associated to
the Hermitian metric $h$ is the unique metric connection
$$\nabla: C^{\infty}(X, L) \to C_1^{\infty}(M, X \otimes T^*X)$$
whose connection
$1$-form in a holomorphic frame $e_L$ has type $(1,0)$.
 The connection
1-form is given by $\nabla e_L = \alpha \otimes e_L$ with $\alpha = \partial \log |h|$.

The metric $g^*$ is a Hermitian metric on $K$. Any Hermitian metric $h$ on a line bundle $L$ induces metrics $h^m = e^{-m \phi}$ on the
tensor powers $L^m$ in the local frame $e_L^m$.
The Hermitian metric and complex structure determine a Chern connection
$\partial \log h$ whose  curvature 2-form $\Theta_h$ is
given locally by
\begin{equation*}
\Theta_h=-\ddbar
\log\|e_L\|_h^2,
\end{equation*}
and we say that $(L,h)$ is positive if the (real) 2-form $\frac{\sqrt{-1}}{2}\Theta_h$ is positive.

\subsection{Curvature form}

Given a connection $\nabla$ on $L$ and a vector field $V$ on $X$,
the covariant derivative of a section $s$ is defined by $\nabla_V s =
\langle \nabla s, V \rangle$. The curvature is the 2-form $\Omega^{\nabla}$
defined by $\Omega^{\nabla}(V, W) = \lbrack\nabla_V, \nabla_W\rbrack - \nabla_{\lbrack V,W\rbrack }.$
If $e_L$ is a local frame and $\nabla e_L = \alpha \otimes e_L$ then
$\Omega^{\nabla} = d \alpha$.

\subsection{Examples}

\begin{itemize}
\item Let $F^*X$ be the unit co-frame bundle $X$, consisting of orthonormal frames of $T^*X$.
Then $S^m F^* X$ is the bundle of real $m$-differentials, i.e., homogeneous polynomials of degree $m$
in $dx, dy$ or $dz, d \bar{z}$.\bigskip

\item When $X$ is given a complex structure, we may decompose $T^*X \otimes \C = T^{*(1,0)} \oplus
T^{*(0,1)}$ into co-vectors $f dz$ of type $(1,0)$ and $g d\bar{z}$ of type $(0,1)$. The holomorphic
tangent bundle is usually denoted by $K = K_X = T^{*(1,0)}$ and is called the canonical bundle. Its
tensor powers $K^m$ are bundles of differentials of type $f (dz)^m$ with $f(z, \bar{z})$ a smooth function.\bigskip

\item When the genus is  $0$, i.e., $X  = S^2$, $K_X$ is a negative line bundle and has no holomorphic sections.
The associated circle bundle of frames is $SO(3) \cong \R {\mathbb P}^3 = S^3/\pm1$.\bigskip

\item When
the genus is $1$, then $K_X$ and $T^*X$ are trivial and $F_h \cong \T^2 \times S^1$. \bigskip

\item When the genus is $ \geq 1$ we may twist $L$ by a flat line bundle. This is not particularly relevant for this
article except that we usually ignore this additional degree of freedom. There is an ample line bundle $L \to \T^2$ whose holomorphic sections are
theta functions. The associated principal $S^1$ bundle is the reduced Heisenberg group, the quotient
of the simply connected Heisenberg group by the integer lattice \bigskip

\item When the genus is $ \geq 2$ then $X = \mathbb{H}^2/\Gamma$ where $\Gamma \subset PSL(2, \R)$. The associated $S^1$ bundle is $SL_2(\mathbb{R})/\Gamma$.  $K_X$ is ample and for $m$ large
there are many holomorphic sections of $K_X^m$. This is only significant in this article when we discuss
splitting eigenspaces.

\end{itemize}

\subsection{Canonical bundle and $m$-differentials}

Let $K \to X$ denote the canonical bundle $T^{*(1,0)} X$ of $(1,0)$ forms
$ f dz$.  Up to twisting by a flat line bundle, it is the unique ample line bundle
on $X$. Hence there exists a Hermitian metric $h$ on $K$ with curvature
form $i \ddbar \log h = \omega$. This should be distinguished from the
curvature of $g$, which is given by $$dd^c \log g^{1 \bar{1}} = K \omega_h, \; \mathrm{where\; K \; is\; the\; scalar\; curvature}.$$ In terms of the Hermitian metric  $h = e^{- \phi_0}$ on $K$,  $|dz|_g = e^{- \phi_0} = g^{1 \bar{1}}$. Also,
$$ \partial_z \log \omega_0 = \partial_z \log (1 - \Delta_0 \phi), \; \omega_{\phi} = (\omega_0 + dd^c \phi) =( (1 - \Delta_{0}) \phi) \omega_0.$$

We now regard $g(X, Y) = \omega(X, JY)$ as a \kahler metric.
The co-metric $^*$ defines  metric coefficients on $T^*X \otimes \C$ by extending $g^*$ by complex linearity
and induces the Hermitian metric,
$$\|d z\|_{g^*} = g^{1 \bar{1}},$$ on $K_X$.
The curvature $(1,1)$ form is therefore $\ddbar \log g^{1 \bar{1}}$.

The associated $(1,1)$ form $\omega_h$ is positive if the genus is $\geq 2$ and if $g$ is a metric of negative
curvature $K$. It is negative if the genus is $0$ and the metric $g$ is of positive curvature. When the genus is $0$ there
do not exist Hermitian metrics with strictly positive or negative curvature forms.

When $\dim X =2$ we write the area form as $dA_g = \sqrt{g} dx$ or as
$\omega_g$. The metric $g$ induces metrics $g$ on $TX, T^*X,  T^{*(1,0)}X = K_X$
and on powers such as $K^m$. On $K^m$, the Hermitian metric induced by
$g$ is $\|dz\|_g^2 = e^{-\phi} = g^{1 \bar{1}}$ so $\phi = - \log g^{1 \bar{1}}$.
The Chern connection on $K$ is the same as the Riemannian connection.

Consider the complex line vector bundles $T^{1,0}_X$ and $(TX, J)$. They
are isomorphic under the map
$$\xi: T_X \to T_X^{1,0}, \;\; v \to \half(v - i J v). $$

\begin{lemm} Let $(X, g)$ be a \kahler manifold. Under the isomorphism $\xi \in T_X \to T_X^{1,0}$, the Chern connection $D$ on the holomorphic tangent bundle $T^{1,0}$
is the Levi-Civita connection $\nabla$. \end{lemm}

\subsection{\label{SPECIALC} Orthonormal frame bundles and $m$-differentials}

If $(X,g)$ is a Riemannian surface, and $F $ is the $SO(2)$-bundle of orthonormal frames of $T^*X$, then
$g$ determines a Riemannian connection on $F$. Similarly, if we fix a complex structure
and define  the principal
$S^1$ bundle $F_h \subset TX$ associated to  $K_X = T^{*(1,0)}$, then $g$ determines a Hermitian
metric on $K_X$. We first discuss the real geometry and then the complex geometry.

The metric $g = \sum_{i,j =1}^2 g_{ij} dx_i \otimes dx_j$ on $TX$ induces a co-metric $g^*$ on $T^*X$, usually denoted by raised indices. It then induces metrics $g^{\otimes n}$ on powers $S^mT^* X$.

There always exists a basis of {\it basic}  or horizontal $1$-forms at a frame
$(\mu_1, \mu_2)$  such that
\[
(\omega_j)_{x, \mu_1, \mu_2} (v) = \mu_j(\pi_* v).
\]

The  Riemannian  connection $1$-form $\alpha$ is defined by the equations,
$$\left\{ \begin{array}{l} \langle \alpha, \frac{\partial}{\partial \theta} \rangle =1, \\ \\
d \omega_1 = \alpha \wedge \omega_2, \\ \\
d \omega_2 = - \alpha \wedge \omega_1, \\ \\
d \alpha = K \omega_1 \wedge \omega_2, \end{array}\right. $$
where $K$ is the scalar curvature. Dually, there exist vector fields $\xi_1, \xi_2$ so that
$$\left\{ \begin{array}{l}[\frac{\partial}{\partial \theta}, \xi_1]= \xi_2 \\ \\

[\frac{\partial}{\partial \theta}, \xi_2]= - \xi_1 , \\ \\

[\xi_1, \xi_2]=  K \frac{\partial}{\partial \theta}  \end{array}\right. $$

Then define
$$E^{\pm} = \half (\xi_1 \mp \xi_2). $$
Then,$$\left\{ \begin{array}{l}[\frac{\partial}{\partial \theta}, E^+]= E^+\\ \\

[\frac{\partial}{\partial \theta}, E^-]= - E^- , \\ \\

[E^+, E^-]= \frac{i}{2}  K \frac{\partial}{\partial \theta}  \end{array}\right. $$

In the frame $\{\xi_1, \xi_2, \frac{\partial}{\partial \theta}\}$ the volume form is $\omega_1 \wedge \omega_2 \wedge \alpha$.
The vector fields $\xi_1, \xi_2$ are also of unit length since they are defined to be horizontal lifts of a unit frame.

 \subsection{ Hilbert spaces of sections}

Let $(L,h) \to X$ be a Hermitian holomorphic line bundle. We thus have a pair
of metrics, $h$ resp. $g$ (with \kahler form $\omega_{\phi}$) on $L$ resp. $TX$.

To each pair $(h,g)$ of metrics we associate Hilbert space inner products $\mathrm{Hilb}_m(h,g)$ on
sections
$s \in L^2_{m \phi}(X, L^m)$ of the form
$$\|s\|_{h^m}^2 : = \int_X |s(z)|^2_{h^m} \omega_{g}, $$
where $ |s(z)|^2_{h^m}$ is the pointwise Hermitian norm-squared of
the section $s$ in the metric $h^m$. In a local holomorphic frame $e_L$, we write
$$\|e_L\|_h^2 = e^{- \psi}. $$
 In local coordinates $z$ and the local frame $e_L^m$
of $L^m$, we may write $s = f e_L^m$ and then
\[
|s(z)|^2_{h^m} = |f(z)|^2 e^{- m \psi(z)} \|e_L\|_{h_0}^{2m}.
\]
Henceforth we write
$$\| f e_L^m\|_{h^m}^2 : = \int_X |f(z)|^2e^{-m \psi} \|e_L\|^{2m}_{h_0}\omega_{\phi}. $$
Locally we may also write $ \|e_L\|^2_{h_0} = e^{-\psi_0}$.

In the special case where $L = K_X$, we may use the frame $dz$ in a  local holomorphic
coordinate $z$. In the  local frames $(dz)^m$
of $K^m$ we may write sections as $s = f (dz)^m$ and then $ |s(z)|^2_{h^m} = |f(z)|^2 e^{- m \psi(z)} \|dz\|_{h_0}^{2m}$ and then,
$$\| f (dz)^m\|_{h^m}^2 : = \int_X |f(z)|^2e^{-m (\psi_0 + \psi)}d A_g, $$
where $dA_g = \omega_{\phi}$ is the area form of $g$.

\section{\label{BOCHNERSECT} Bochner Laplacians on line bundles}

In this section, we give explicit local formulae for Bochner Laplacians
 Bochner Laplacians $\nabla^*_{h,g} \nabla$ on $L^2(X, L)$
equipped with the data
\[
(g, h, J, J_L, \nabla),
\]
where $(L, h) \to X$ is a Hermitian holomorphic  line bundle, $g$ is a metric
on $X$, $\nabla$ is a connection on $L$. In a local frame $e_L$ of $L$, with $s = f e_L$, the inner product  $\mathrm{Hilb}(g,h)$ on $L^2(X, L)$ takes the
form,
$$\|s \|^2_{\mathrm{Hilb}(g,h)} = \int_X |f|^2 e^{-\psi} dV_g, \;\; (\mathrm{where}\; \|e_L(z)\|_h^2 = e^{-\psi(z)}).$$
The inner product on $L^2(X, L \otimes T^* X)$ has the form,
$$\|s\otimes \eta \|^2_{\mathrm{Hilb}(g,h)} = \int_X |f|^2 \|\eta\|_g^2 e^{-\psi} dV_g, \;\; (\mathrm{where}\; \|e_L(z)\|_h^2 = e^{-\psi(z)}).$$
With no loss of generality, we fix $J$ on $X$ and assume that $(g, J, \omega)$ is a \kahler metric with $g(X, Y) = \omega(X, J Y)$.
Then  $g(\frac{\partial}{\partial z}, \frac{\partial}{\partial z} ) = 0 = g(\frac{\partial}{\partial \bar{z}}, \frac{\partial}{\partial \bar{z}})= 0$. There is only one metric coefficient, $g^{1 \bar{1}} = G(dz, d \bar{z})$. It is a Hermitian metric on $T^{1,0} X$
    and is compatible with $J$. \footnote{Although $\dim M =2$, we use
 the notation $dV_g$ and the term `volume form' to avoid clashing with the notation $A$ for connections and area form.} We also
 denote the Riemannian volume form by $dV_g = \omega = dd^c \log g^{1 \bar{1}}$.

\begin{rema} Notational remark: We use $G$ rather than $g^{-1}$ or $g^*$ for the dual co-metric on $1$-forms, because it is a convenient
notation for later variations.  \end{rema}

   The Bochner Laplacian is the Laplacian on $L^2(X, L)$
determined by the quadratic form,
\begin{equation}\label{q} q_{g, h, \nabla}(s) = \int_X |\nabla  s|^2_{h \otimes g} dV_g =\langle \nabla_{g,h}^* \nabla s,s \rangle_{\mathrm{Hilb}(g,h) }.  \end{equation}
Throughout we assume that $g$ is $J$-compatible. In a local frame $e_L$ of $L$, with $s = f e_L$, $\nabla (f e_L) = (df + f \alpha)\otimes e_L$
 and with $\|e_L(z)\|_h^2 = e^{- \psi(x)}$,
 and the quadratic form is given by
 \begin{equation*} q_{g, h, \nabla}(f e_L) = \int_X |df + f \alpha|^2_{ g} e^{-\psi} dV_g.  \end{equation*}
 The adjoints are  taken with respect to the volume form $e^{-\psi} dV_g$.

 We give local formulae for  $\nabla^*_{h,g} \nabla$ under several assumptions on $\nabla$ and in correspondingly adapted
 frames (equivalently, choosing a gauge for $\nabla$):

  \begin{itemize}
  \item[(i)]  $\nabla$ is $h$-compatible (see Section \ref{UNITSECT}); in this case, we compute in  a local unitary frame.
  Fixing $h$ is equivalent to fixing the principal $S^1$ bundle $P_h \to X$, and varying the connection $1$-forms
  $\alpha \in \acal_h$  on $P_h$ (with fixed $g$). \bigskip

  \item[(ii)] $\nabla \in \acal_{\C}$ is compatible with a  fixed complex structure $J_L$ on $L$ (see Section \ref{HOLOSECT}); in this
  case we compute in  a local holomorphic frame. In the next section we fix $J_L$ and vary $h$ (with fixed $g$).    \bigskip

  \item[(iii)] $\nabla$ is compatible with both $(h, J_L)$, hence is the Chern connection; see Section \ref{CHERNSECT}.
  The $(1,0)$-part of the connection has the form $\alpha =\partial \log h$ and is parameterized by the Hermitian metric $h$ on $L$;  in the next section we consider its
  variation with $h$ (with fixed $g$).\bigskip

  \item[(iv)] When $L = K_X $ is the canonical line bundle, we let $h$ be the Hermitian metric induced by $g$ and
  let $\nabla$ be the Levi-Civita connection. This is a special case of an $h$-compatible connection but is special
  because $h$ is induced by $g$. Moreover, the Riemannian connection w.r.t. $g$ is the Chern connection for the Hermitian metric $g$.
    In a holomorphic frame $dz$ the connection form is $\partial \log g^{1 \bar{1}} = \partial \log G(dz, d \bar{z})$. Also,
 $dV_g = g^{1 \bar{1}} dz \wedge d \bar{z}$.   In the next section, we vary this connection by varying $g$ on $X$.

   \end{itemize}
 \bigskip


There exist many formulae for Bochner Laplacians in the literature (see for instance \cite{D}), but they often make assumptions on the compatibility of the connection
with other data (the Hermitian metric or complex structure) and  we need explicit dependence on the compatibility
conditions so that  we can perturb some of the data while holding others fixed. We therefore go through the calculations with explicit assumptions on the compatibility
of $\nabla$ with the data  $(g, h, \nabla, J, J_L). $

We also recall the general identities,
 $d^*(f \alpha) = - * d * (f \alpha) = - * d (f* \alpha) = - * df \wedge * \alpha + f d^*\alpha.$ Note that $G (\eta, \zeta)\omega = \eta \wedge * \zeta$.
Hence, $- * df \wedge * \alpha = -* G(df, \alpha) \omega= - G(df, \alpha)$ since $* \omega = 1$.

\subsection{\label{UNITSECT}Calculation in a unitary frame}
In this section we assume that $\nabla$ is compatible with $h$.
  We recall from Section \ref{GEOMSECT} that on  a Hermitian  line bundle $(L,h)$, the set  $\acal_h$ of connections on $L $
which are compatible with the Hermitian metric is the affine space $\{A_{\alpha} = A_0 + \alpha: \alpha \in \Omega^1(X)\}$ where $A_0$ is a fixed background
connection and $\Omega^1(X)$ are the real $1$-forms on $X$.  The Hermitian metric determines the principal $U(1)$
bundle $P_h$ of unitary frames of $L$ and as before $A_1$ determines a connection $1$-form $\alpha_1$ on $P_h$. On the base $X$, the connection $1$-form is $i \R$-valued
in a unitary frame and we write it as $i \alpha$ with a real-valued $\alpha$.

\begin{prop} \label{BOCHUNIT}Let $(X, g)$ be a Riemannian manifold and let $(L, h)$ be a
 Hermitian line bundle with $h$-compatible connection $\nabla$. Let
$\nabla (f e_L) =(d f + i f\alpha)\otimes e_L$ with $\alpha \in  \R$ in a unitary frame $e_L$. Then,
  $$  \nabla^* \nabla (f e_L)  = \left(- \Delta_g f   - 2 i  G(df, \alpha)   + if d_g^*\alpha + G(\alpha, \alpha)  f\right) e_L.$$
  where $\Delta_g $ is the scalar Laplace operator.
\end{prop}

\begin{proof}  In a unitary frame,  $|e_L|_h^2 = e^{-\psi}=1$ and this
factor drops out. We leave it in until the last step for purposes of later comparison to other frames. Since
$\nabla (f e_L) =  df \otimes e_L + i f \alpha e_L$,  and by  \eqref{q},
\[
 q_{g, h, \nabla}(s)  = \int_X |d f + i f\alpha |_g^2 e^{-\psi}   dV_g.
\]
 Note that
\[
|d f + i f\alpha|^2_{ g} = G(df + i f \alpha, \overline{df + i f \alpha}) =  |d f|^2_g  +  2 \Re \bar{f }  G(d f, -i \alpha) + G(\alpha, \alpha) |f|^2
\]
is the Hermitian norm-squared, so
\begin{align*}
&q_{g, h, \nabla}(s)
 =&  \int_X \left( |d f|^2_g  +  2 \Re \bar{f }  G(d f,- i \alpha) +   G(\alpha, \alpha) |f|^2\right) e^{-\psi}   dV_g\\
  =&   \int_X \left( -2 \Re \bar{f } G(d f, i\alpha) +  G(\alpha, \alpha) |f|^2\right) e^{-\psi}   dV_g - \int_X \bar{f} d_g^*(e^{-\psi}df) dV_g\\
 = &  \int_X \left( -2 \Re \bar{f }  G(d f, i \alpha) + G(\alpha, \alpha) |f|^2\right)  dV_g  - \int_X \bar{f} (d_g^* df)  dV_g,
\end{align*}
   where in the last line we use that $\psi = 0$ in a unitary frame. Since $\alpha$ is real-valued,
\[
-2\Re \bar{f }  G(df, i \alpha) = -i( \bar{f }  G(df, \alpha) -  f G(d \bar{f}, \alpha) ).
\]
   Recall that $d_g^*(f \alpha) = - * d * (f \alpha) = -  G(df, \alpha) + f d_g^*\alpha$. Replacing
   $ i G(d \bar{f}, \alpha)$ by $-i \left(d_g^*(\bar{f} \alpha)  - \bar{f} d_g^*\alpha\right)$ and integrating the $d_g^*$ by parts gives
\begin{multline*}
\int_X \left( - i  \bar{f }  G(df, \alpha) -i  \bar{f} G(df,  \alpha)   + i |f|^2 d_g^*\alpha + G(\alpha, \alpha) |f|^2\right)\;  \;  dV_g\\
 - \int_X \bar{f} (d_g^* df)  dV_g,
\end{multline*}
   Thus, we get
\[
\nabla^* \nabla f = - \Delta_g f   - 2 i  G(df, \alpha)   + if d_g^*\alpha + G(\alpha, \alpha)  f. \qedhere
\]

 \end{proof}

\subsection{\label{HOLOSECT} Holomorphic line bundles: $J_L$-compatible connections. }

In this section we give a local formula for the Bochner Laplacian when $L \to X$ is a holomorphic line bundle and $\nabla$ is compatible
with the complex structure. Thus, complex
structures $J$ on $X$ and $J_L$ on $L$ are fixed. In a holomorphic frame,   $\|e_L \|= e^{-\psi} \not=1$ and $\nabla e_L = \alpha \otimes e_L$,
where  $\alpha $ is
of type $(1,0)$. We write $\nabla = \partial^{\nabla} + \dbar^{\nabla}$ for the decomposition of a connection into its $(1,0)$ resp. $(0,1)$ parts,
with $ \partial^{\nabla}  = \nabla^{(1,0)}, \dbar^{\nabla} = \nabla^{(0,1)}.$

 The Bochner-Kodaira identity relates $\nabla^* \nabla$ to $\dbar_L^* \dbar_L$, where
$$\left\{ \begin{array}{l} \partial_L (fe_L) : =\nabla^{(1,0)} (fe_L) = (\partial f +\alpha  f) \otimes e_L, \\ \\\dbar_L(fe_L): =
\nabla^{(0,1)} (fe_L) = \dbar f \otimes e_L. \end{array} \right.$$

The analogue of Proposition \ref{BOCHUNIT} is

    \begin{prop} If $\nabla$ is compatible with $J_L$ with connection $1$-form $\nabla e_L = \alpha \otimes e_L$ with
    $\alpha$  of type $(1,0)$ in the holomorphic frame $e_L$, then
     $$  \nabla^* \nabla (f e_L) = \left(  - \Delta_g f   + G(d \psi  + \alpha, df)
  - f G(d  \psi, \alpha)    + f d^*\alpha + G(\alpha, \bar{\alpha}) f
     \right) e_L.$$
     \end{prop}

\begin{proof} The proof is similar to that of Proposition \ref{BOCHUNIT}, with two differences: (i) We use a holomorphic frame
rather than a unitary frame and $|e_L|_{h}^2 = e^{-\psi}$ is not equal to $1$; (ii)  $\alpha$ is of type $(1,0)$ rather being
$i \R$-valued.  Note that $$ |d f +  f\alpha|^2_{ g} = G(df +  f \alpha, \overline{df +  f \alpha}) =  |d f|^2_g  +  2 \Re \bar{f }  G(d f, \bar{ \alpha}) + G(\alpha, \bar{\alpha}) |f|^2$$

  By \eqref{q}, and integrating
by parts the $|df|^2_g$ term, and with $|\cdot|^2$ denoting the Hermitian metric,  we get
\begin{align*}
 &q_{g, h, \nabla}(s) \\
 = &
 \int_X |d f + f\alpha |_g^2\; e^{-\psi} \;  dV_g \\
  = & \int_X \left( |d f|^2_g  +  2 \Re \bar{f }  G(d f,  \bar{\alpha}) + G(\alpha, \bar{\alpha})  |f|^2\right)\; e^{-\psi} \;  dV_g\\
  = &  \int_X \left( 2 \Re \bar{f } G(d f, \bar{\alpha}) + G(\alpha, \bar{\alpha}) |f|^2\right)\; e^{-\psi} \;  dV_g
 - \int_X \bar{f} d_g^*(e^{-\psi}df) dV_g \\
  = &  \int_X \left( 2 \Re \bar{f }  G(d f,  \bar{\alpha}) + G(\alpha, \bar{\alpha}) |f|^2\right)\; e^{-\psi} \;  dV_g
 - \int_X \bar{f} (d_g^* df - G(d \psi, df)) e^{-\psi} dV_g
\end{align*}

Further,
\[
2\Re \bar{f }  G(df,  \bar{\alpha}) =   \bar{f }  G(df, \bar{\alpha}) +  f G(d \bar{f}, \alpha).
\]
   We simplify the $f G(d \bar{f}, \alpha) $ term using
   that $d^*(\bar{f} \alpha) = - * d * (\bar{f} \alpha) = -  G(d\bar{f}, \alpha) + \bar{f} d^*\alpha$ so that
   $ G(d \bar{f}, \alpha) =- d^*(\bar{f} \alpha)  +  \bar{f} d^*\alpha.$
Integrating the $d^*$ by parts gives
\begin{align*}
  &\int_X \left( 2 \Re \bar{f }  G(d f,  \bar{\alpha}) + G(\alpha, \bar{\alpha}) |f|^2\right)\; e^{-\psi}   dV_g\\
    = & \int_X \left(   \bar{f }  G(df, \bar{\alpha}) +  f G(d \bar{f}, \alpha) + G(\alpha, \bar{\alpha}) |f|^2\right) e^{-\psi} \\
   =&  \int_X \left(  \bar{f }  G(df, \bar{\alpha}) - f d^*(\bar{f}\alpha)   + |f|^2 d^*\alpha +G(\alpha, \bar{\alpha}) |f|^2\right)   e^{-\psi}dV_g\\
    =&  \int_X \left( \bar{f }  G(df, \bar{\alpha}) - \bar{f} G(df, \alpha) - |f|^2 (G(d \psi, \alpha)       + |f|^2 d^*\alpha +G(\alpha, \bar{\alpha}) |f|^2\right)   e^{-\psi}dV_g.
\end{align*}

   Combining with the term $ - \int_X \bar{f} (d_g^* df - G(d \psi, df)) e^{-\psi} dV_g $,  we get
 \[
 \nabla^* \nabla f = - \Delta_g f   + G(d \psi  + \bar{\alpha} - \alpha, df) - f G(d  \psi, \alpha)    + f d^*\alpha + G(\alpha, \bar{\alpha}) f \qedhere
 \]
 \end{proof}

 \subsection{\label{CHERNSECT} Chern connection}
 In this section we assume $\nabla$ is both $h$-compatible and $J_L$-compatible, i.e., that
 it is the Chern connection with connection $1$-form $\partial \psi$.
 One can then compute   $\nabla^*\nabla$ using the  relation
 \begin{equation}\label{BOCHKOD} \nabla^* \nabla = 2 \dbar^{\nabla *} \dbar^{\nabla} + i *d \alpha \end{equation} between
the Kodaira and Bochner Laplacians.

Note that $d \psi$ is real and $\partial \psi = \alpha$, so $d \psi = \alpha + \bar{\alpha}$ and
$G(d \psi + \alpha - \bar{\alpha}, df) = G(\alpha, \dbar f)$ above. Also, $G(d \psi, \alpha) =G(\bar{\alpha}, \alpha)$,
so the terms $- f G(d\psi, \alpha) + G(\alpha, \bar{\alpha}) f$ cancel and from the preceding Proposition we
get $$\nabla^* \nabla (f e_L) = -\Delta_g f + G(\alpha, \dbar f) + f d_g^* \alpha. $$

We now prove this directly.

\begin{prop} \label{CHERN} Let $\nabla$ be the Chern connection for $(L,h)$. Then,
$$\nabla^*\nabla (f e_L)=  (-\Delta_g f +  G(\partial \psi, \dbar f) +  f (i * \Omega^{\nabla} ))e_L. $$

\end{prop}

\begin{proof}

  Using \eqref{BOCHKOD} and $\dbar^{\nabla}(f e_L) = (\dbar f \otimes e_L)$, we have

 $$\begin{array}{lll} \langle \dbar_L^* \dbar_L s, s \rangle_h &&:= \langle \dbar_L s, \dbar_L s \rangle_{h \otimes g}. \end{array} $$

We rewrite $ G(\dbar f, \partial f)  $ term using
   that $d^*(\bar{f} \partial f) = - * d * (\bar{f} \partial f) = -  G(d\bar{f}, \partial f) + \bar{f} d^*\partial f$ so that
   $ G(\dbar f, \partial f) =- d^*(\bar{f} \partial f)  +  \bar{f} d^*\partial f.$
Integrating the $d^*$ by parts gives
\begin{align*}
\int_X G(\dbar f, \partial f) e^{-\psi} \omega_g  &=   \int_X ( - d^*(\bar{f} \partial f)  +  \bar{f} d^*\partial f) e^{-\psi} \omega_g\\
 &=  \int_X \left( G( \partial f, d \psi) - \Delta f) \right) \bar{f} e^{-\psi} \omega_g  .
\end{align*}
Adding the curvature term adds $ f(i * \Omega^{\nabla})$.\
\end{proof}

\begin{rema} Proposition \ref{CHERN} and \eqref{BOCHKOD} are consistent by the following
calculation: If $\alpha = \partial \psi$ is a Chern connection $1$-form, then $$(d^*_g \alpha) \omega_g
= \Omega^{\nabla}= i\ddbar \psi = (\Delta_g \psi) \omega_g,\;\; d^*_g \alpha = \Delta_g \psi. $$
Indeed, in terms of the Hermitian inner product,
\begin{align*}
 \langle d_g^* \alpha, f \rangle_{L^2} =
 \langle \alpha, df \rangle_{L^2} =
 \langle \alpha, \partial f \rangle_{L^2}=
 \int_X \partial \alpha \cdot \dbar \bar{f} \omega_g  &=
 i \int_X \frac{\partial \psi}{\partial z} \overline{\frac{\partial f}{\partial z}} dz
 d\bar{z}\\
   &=  - i \int_X \frac{\partial^2 \psi}{\partial z \partial \bar{z}} \overline{f} dz
 d\bar{z} \\
& = -\langle \Delta \psi, f \rangle_{L^2}.
\end{align*}
\end{rema}

 \subsection{ Canonical bundle:  $L = K$ and $\nabla $ is the Riemannian connection}

Let $z$ be a local holomorphic coordinate and let $dz$ be the associated section of $K$.
Differentials of type $(dz)^m$ are sections of $K^m$, the $m$-th power of the canonical bundle. The Riemannian metric on $X$ induces a Hermitian metric $h^m$ on $K^m$, namely
$|dz|_h = |dz|_g $ where $g$ is the co-metric. $dz = dx + i dy$ and at $x + i y$, $|dz| = y$.


The metric $g$ on $TX$ endows a  Hermitian metric $g^*$ on $K$ and the associated Riemannian  connection $\nabla_{g}$
is the Chern connection with connection $1$-form $\alpha = -  \partial \log g^{1 \bar{1}}$ in the frame $dz$. For
simplicity of notation we write $\phi= - \log g^{1 \bar{1}}$.    It induces
connections and Hermitian metrics on $K^m$ with connection $1$-forms $m \alpha$. The associated  Bochner
Laplacian $\nabla_{m,g}^* \nabla_{m,g}$ on $K^m$ corresponds to
the quadratic form
$$\begin{array}{lll} q_{m, g}(s) & = &  \int_X |\nabla_{m, g} s|^2_{m,g} \omega_g =
 \int_X |d f + m f \partial \phi |^2  \|(dz)^m\|^{2}_{g^m} \omega_g \\&&\\
 & = &  \int_X |d f + m f \partial \phi |^2 e^{- m \phi}  \omega_g. \end{array}$$
 Note that $\omega_g = \frac{i}{2} g_{1 \bar{1}} dz \wedge d \bar{z}$ and
 the Laplacian on scalar functions  is given by
$\Delta_0 f =  g^{1 \bar{1}} \frac{\partial^2 f}{\partial z \partial \bar{z}}. $

\begin{prop}  Let $\nabla_m$ be the Chern connection for $(K^m, g^{*m})$. Then,
$$\nabla_m^*\nabla_m (f (dz)^m)=  (-\Delta_g f +  m G(\partial \phi , \dbar f) + m K f ) (dz)^m. $$
\end{prop}

\begin{proof} This follows from Proposition \ref{CHERN}. We give a direct proof. By the Bochner-Kodaira formula \eqref{BOCHKOD}, it suffices to prove

\begin{equation} \label{ddbarprop}
\dbar_m^* \dbar_m (f (dz)^m)  =  \left(g^{1 \bar{1}} \frac{\partial^2 f}{\partial z \partial \bar{z}} - m \left(\frac{\partial f}{\partial \bar{z}} g^{1 \bar{1}}\right)  \frac{\partial \phi}{\partial z} \right) (dz)^m,
\end{equation}
where $\phi(z) = - \log |dz|_g= - \log g^{1, \bar{1}}.$

As above, we calculate the adjoint to be
\begin{multline*}
\dbar_m^*(f (dz)^m \otimes d\bar{z}) = \left(e^{m\phi} \omega_{h}^{-1}  \frac{\partial}{\partial z}\left(f(z) g^{1 \bar{1}} e^{-m \phi} \omega_h\right) \right) \\
= g^{1 \bar{1}}   \frac{\partial f}{\partial z}  - m f(z) g^{1 \bar{1}}   \frac{\partial \phi}{\partial z}.
\end{multline*}

It follows that
\begin{align*}
&\dbar_m^* \dbar_m (f (dz)^m)  \\
=&  \dbar_m^*(\frac{\partial f}{\partial \bar{z}} (dz)^m \otimes (d \bar{z})\\
  =&  \left(e^{m\phi} \omega_{g}^{-1}  \frac{\partial}{\partial z} \left(\frac{\partial f}{\partial \bar{z}} g^{1 \bar{1}} e^{-m \phi} \omega_g\right) \right) \\
  =&  g^{1 \bar{1}} \frac{\partial^2 f}{\partial z \partial \bar{z}} - m \left(\frac{\partial f}{\partial \bar{z}} g^{1 \bar{1}}\right)  \frac{\partial \phi}{\partial z} + \frac{\partial f}{\partial \bar{z}} g^{1 \bar{1}} \frac{\partial \log  g^{1 \bar{1}}}{\partial z}
+  \left(\frac{\partial f}{\partial \bar{z}} g^{1 \bar{1}}\right)  \frac{\partial \log \omega_g}{\partial z} \\
  = & g^{1 \bar{1}} \frac{\partial^2 f}{\partial z \partial \bar{z}} - m \left(\frac{\partial f}{\partial \bar{z}} g^{1 \bar{1}}\right)  \frac{\partial \phi}{\partial z} ,
\end{align*}
where we used $ \frac{\partial \log  g^{1 \bar{1}}}{\partial z} +  \frac{\partial \log \omega_h}{\partial z} = 0. $
\end{proof}

\section{\label{PERTURBATIONSECT} Perturbation theory and genericity}

In this section we prove generic properties of the eigenvalues and
eigensections of Bochner Laplacians $\nabla_{g,h}^* \nabla$
on complex holomorphic Hermitian  line bundles $(L, h) \to X$.  Our ultimate goal
is to deduce generic properties of Kaluza-Klein Laplacians on
the principal $U(1)$ frame bundles $P_h \to X$ associated to $h$. First
we discuss generic properties of Bochner Laplacians on the line bundles
and then we draw conclusions for the  Kaluza-Klein Laplacians. We prove
that for generic data $(g, h, \nabla)$ (with fixed $(J_L, J)$), eigenvalues
of Bochner Laplacians $\nabla_{g,h}^* \nabla$ are simple (multiplicity one)
and all eigensections intersect the zero section transversally (i.e., have
$0$ as a regular value). This immediately implies that for the associated Kaluza-Klein Laplacians $\Delta_G$  on $P_h$, all joint eigenfunctions of the $U(1)$
action and $\Delta_G$ have simple joint spectrum and have $0$ as a regular value. In Section \ref{simple}, we discuss the multiplicity of the spectrum of $\Delta_G$, hence proving a part of Theorem \ref{MAIN}.

The main result of this section is:

\begin{theo}\label{GENERIC1}

For generic `admissible data' described below, and for every $m$,  the spectrum of
each Bochner Laplacians $\nabla_{g,h}^* \nabla$ on $C^k(X, L^m)$ is simple  and  all of its eigensections have zero as a regular value.
 Moreover, if we lift sections to equivariant eigenfunctions $\phi$,
then $\Re \phi$ and $\Im \phi$ have zero as a regular value.

 The generic admissible data is of the following kinds:

  \begin{itemize}
  \item[(i)]  We fix $h,g$ and vary the connection  $\nabla$ in $\acal_h$.
  Fixing $h$ is equivalent to fixing the principal $U(1)$ bundle $P_h \to X$, and varying the connection $1$-forms. \bigskip

  \item[(ii)] We fix $(J, J_L, g)$ and vary both  $h$ and $\nabla$, assuming that $\nabla \in \acal_{\C}$ is compatible with  $J_L$ on $L$ but not necessarily with $h$.    \bigskip

  \item[(iii)] We fix $(g, J_L, J)$ and vary $(h, \nabla)$ assuming that $\nabla$ is compatible with both $(h, J_L)$, hence is the Chern connection
  of $(L, J_L, h)$.\bigskip

  \item[(iv)] We fix $L = K^m$ and also fix $J$ and vary $g$ in the conformal class  associated to $J$. We assume that $h$ is the Hermitian metric induced by $g$ and
that  $\nabla$ is the Levi-Civita connection.

   \end{itemize}

   \end{theo}

   The proofs in each of the cases are given in separate sections.

   Note that
the functions relevant to this article are smooth sections of a complex line bundle $L$,
and may locally be represented as complex valued functions $u$. We will
prove that
 $u: M \to \C$ has zero as a regular value, i.e., that  $d u_p
= d \Re u + i d \Im u$
is surjective. It follows that $\Re u, \Im u$ are independent and nowhere
vanishing on their zero sets, and that
each has zero as a regular value

\subsection{The Uhlenbeck framework}

To study generic properties of the spectrum, we follow \cite{U} and work with $C^r$ spaces of metrics and connections. We use the following notation:

\begin{itemize}
\item We denote by $\gcal^r(X)$ the Banach space of $C^r$ metrics on $X$.
Since $X$ is a surface and we usually fix the complex structure $J$, we
only work with $C^r$ metrics in the associated conformal class $\mathrm{Conf}(J)$ and represent them in the usual Weyl gauge $g = e^{\rho} g_0$
relative to a fixed background metric $g_0 \in \mathrm{Conf}(J)$. Thus,
we may identify $\gcal^r(X) \cong C^r(X)$. We may also fix the
area of the metrics with no loss of generality and then $\mathrm{Conf}(J)$
may be identified with the space $\kcalomega$ of \kahler metrics
on $X$ in a fixed cohomology class. This is simply a different choice
of gauge in which we write the \kahler forms as $\omega_{\phi} = \omega_0 + i \ddbar \phi$ and use the potentials $\phi$ rather than the Weyl gauge $u$ to parameterize metrics. \bigskip

\item We denote by $\hcal^r(L)$ the Banach space of $C^r$ Hermitian
metrics on $L$. Once we fix a local frame $e_L$ we may identify
$h \in \hcal^r(L)$ with the function $\psi$ such that $\|e_L(z)\|^2_h = e^{-\psi(z)}$, and $\hcal^r(L) $ is then equivalent to $C^r(X)$ except of
course that the identification is frame dependent and the frame is only local (defined on the complement of a smooth closed  curve in $X$, e.g.).

\bigskip

\item We denote by $\acal^r(L)$ the space of connections with
$C^r$ connection forms. As before, we also denote by $\acal_h^r$, resp.
$\acal_{\C}^r$,  the
$h$-compatible (resp. $J_L$-compatible) $C^r$ connections. \bigskip

\item We denote by $C^r(X, L)$ the $C^r$ sections of $L$. We also denote by $H^s(X, L)$ the Sobolev space of sections with $s$ derivatives
in $L^2$.
\bigskip

\end{itemize}

We  define
\begin{equation*} \Phi_L: \gcal^r(X) \times \hcal^r(L) \times \acal^r(L) \times H^2(X, L) \times \C  \to
L^2(X, L), \;\; \end{equation*}
by
\begin{equation*}
\Phi_L(g, h, \nabla, s, \lambda) = (\nabla^*_{g, h} \nabla -\lambda) s.
\end{equation*}
Here, the  eigenvalue parameter $\lambda$ in the domain is allowed to be complex even though at
zeros of $\Phi_L$ it is always real. This does not change the arguments in \cite{U} but is needed so
that $\lambda s$ spans the eigenspace when $s$ is an eigensection. In \cite{U} the eigenfunctions were
real-valued, so this issue did not arise.

Recall that a linear map between Banach spaces is {\it Fredholm} if it has closed
image and finite dimensional kernel and cokernel. The index of a Fredholm operator is the difference of the dimensions
of its kernel and cokernel. A nonlinear map $\Phi: N \to Y$ of Banach manifolds is Fredholm if its derivative $D \Phi_n$ is Fredholm for every
$n \in N$.

Our first goal, roughly speaking, is to prove that $\Phi$ is a Fredholm map of index $0$,
i.e., to prove  surjectivity of the differentials $D_2 \Phi$
from tangent spaces of
$$Q = \{(g, h, \nabla, s, \lambda): \Phi_L(g, h, \nabla, \lambda) = 0\} $$
to $L^2(X, L)$.  It is sufficient to pick the relevant types of frames
and calculate the Bochner Laplacians in the frame as in Section \ref{BOCHNERSECT}.

  Regarding the surjectivity, we need to prove density of the
image and that the image is closed.  Some care needs to be taken because sections of complex line bundles are `vector-valued', i.e., have two real components. As explained in \cite{EP12}, there are pitfalls to avoid
  when generalizing the arguments of \cite{U} to the vector-valued case.
  But sections of line bundles are locally complex-valued functions and are essentially
  scalar functions, albeit with scalars in $\C$.

\subsection{Uhlenbeck's argument} We briefly
review Uhlenbeck's
 proof  that for generic metrics on compact $C^r$Riemannian manifolds,
 all eigenvalues are simple and  all eigenfunctions have $0$ as a regular value.

 Her framework is quite general and therefore uses the notation $B$
 for the relevant space of metrics or other geometric data, and $L_b$
 for the Laplacian associated to $b$. The relevant  functions
 are denoted by $u$ and the space of such functions on a manifold
 $M$ is denoted by $C^k(M)$,
 even though they could be sections of a bundle over $M$.
 Then define
$$\Phi(u, \lambda, b) = (L_b + \lambda) u, $$
and put

\begin{itemize}

\item $Q := \{(u, \lambda, g) \in C^k(X) \times \R_+ \times B: \Phi(u, \lambda, b) = 0\}$.
\bigskip

\item $\alpha: Q \times M \to \C: \alpha(u, \lambda, b, x) = u(x). $\bigskip

\item  $\beta: Q \times M \to T^*M: \beta (u, \lambda, b, x) = \nabla u(x). $\bigskip

\end{itemize}
Then,
\begin{multline*}
T_{u, \lambda, b} Q  = \\
 \{(v, \eta, s) \in H^{1,0}(X) \times \mathbb{C} \times T_b B: \int_X u v dV_g  = 0, \; (L_b + \lambda) v + \eta u + D_2 \phi s = 0\}.
\end{multline*}
We often write $$v = \dot{u}, \eta = \dot{\lambda}, D_2(\Phi) s = \lambda \dot{\Delta} u, \;\; (\Delta + \lambda) \dot{u} + (\dot{\Delta} + \dot{\lambda}) u = 0.$$ Further, let $D_1 \alpha$ denote the derivative of $\alpha$ along $Q$. Then,
$$ D_{(u, \lambda, b)} \alpha(v, 0, c, 0) = v(x) = \dot{u}(x). $$
Also define $J$ to be the image of $D_2 \Phi$,  $$J = \mathrm{Im} D_2 \Phi_{(u, \lambda, b)}  = \{  \dot{\Delta} u:  \dot{\Delta} \; {\mathrm is\;a\;variation\; of} \; \Delta\; {\mathrm along\;a\;curve\;of\;metrics}\}.  $$

We use the following `abstract genericity'  result of \cite[Theorem 1]{U}- \cite[Lemmas 2.7-2.8]{U}.

\begin{theo} \label{U} Assume that $\Phi$ is $C^k$ and has zero as a regular value.  Then the eigenspaces
of $L_b$ are one-dimensional.
If additionally, $\alpha: Q \times M \to \mathbb{C}$ has zero as a regular value,  then additionally
$$\{b \in B: {\mathrm the \; eigenfunctions\; of \;} L_b\; {\mathrm have\; 0\; as\; a \; regular \; value} \} $$
 is  residual in  $B$.
\end{theo}

The key proposition is the following procedure for verifying the first  hypothesis of Theorem \ref{U}.
  (see  \cite[Proposition 2.10]{U}). \begin{prop}\label{PROP}
Let $J = \mathrm{im} D_2 \Phi$ and assume that for $W \in L^1(M)$ and
$W \in C^2(M - \{y\})$, the property $\int_M W(x) j(x) d\mu_x = 0$ for all $j \in J$
implies $W = 0$. Then  $\phi$ is $C^k$ and has zero as a regular value. \end{prop}

For the sake of completeness, we  briefly review the main steps in proving Theorem \ref{U}:
The main input are two transversality theorems. The first is:  Let $\phi: H \times B \to E$
be a $C^k$ map where $H, B, E$ are Banach manifolds. If $0$ is a
regular value of $\phi$ and $\phi_b(\cdot): = \phi(\cdot, b)$ is a Fredholm
map of index $< k$, then the set $\{b \in B: 0 \;{\mathrm is\;a\; regular\;value\;of}  \; \phi_b\}$ is residual in $B$.

The second statement follows from  \cite[Lemma 2.7]{U}:  Let
 $\pi: Q \to B$ be a $C^k$ Fredholm map of index $0$. Then if
$f: Q \times X \to Y$ is a $C^k$ map for $k$ sufficiently large and if $f$
is transverse to $Y'$ then $\{b \in B: f_b: = f|_{\pi^{-1}(b)} \; {\mathrm is\; transverse\; to\;} Y'\}$ is residual in $B$.
Let $$\alpha: f^{-1}(Y') \to B \; {\mathrm be}\; \alpha: f^{-1}(Y') \subset Q \to B.$$

\begin{lemm} The eigenfunctions of $L_b$ have zero as a regular value
if  $b$ is a regular value of $\pi$ and if $0$ is a regular value of
$\alpha |_{\pi^{-1}(b)} \times M : = \alpha_b$.  \end{lemm}

 Eigenfunctions and eigenvalues move continuously under perturbations of the operator. So it is easy to show that the set of metrics with for which the $j$th  eigenvalue is simple
is open. The difficulty is to prove that this set is dense.

To prove the first statement in Theorem \ref{GENERIC1} we need to verify the hypotheses of Theorem \ref{U} and therefore
need to prove Proposition \ref{PROP}, i.e., to determine the range  of $D_2 \phi$.

\begin{prop} \label{SURJ2} For each of the admissible types of perturbation,
$D_2 \Phi_m$ is surjective from $T_{(u, \lambda, \phi)} Q_m \to C^{k-2}.$
\end{prop}

\subsection{Base metric variations } In this section we fix $(h, \nabla, J, J_L)$
and vary only $g = e^{\rho}$. Equivalently, we consider Kaluza-Klein metrics
on a fixed $U(1)$ bundle$P_h \to M$ with a fixed connection $\alpha$ and
vary the base metric $g$.

\begin{prop}

Suppose that     $(L, h, J) \to X$ is a
 Hermitian holomorphic  line bundle with $h$-compatible connection $\nabla$. Let
$\nabla (f e_L) =(d f + if \alpha) \otimes e_L$ with $\alpha \in  \R$ in a unitary frame $e_L$.
Then for generic Riemannian metrics $g = e^{\rho} g_0$ in the conformal
class of $J$, all of the eigenvalues of $\nabla_{g,h}^* \nabla_h$ are
simple and all of the eigensections have $0$ as a regular value. \end{prop}

\begin{proof}
By Proposition \ref{BOCHUNIT},
  $$  \nabla^* \nabla (f e_L)  = \left(- \Delta_g f   - 2 i  G(df, \alpha)   + if d_g^*\alpha + G(\alpha, \alpha)  f\right) e_L.$$
  where $\Delta_g f$ is the scalar Laplace operator, where $g = e^{\rho}$.

 Taking the variation $\delta$ with respect to $\rho$ (and designating
 the variation with a dot),
   $$  \delta \nabla^* \nabla (f e_L)  = \left(- \dot{\Delta}_g f   - 2 i  \dot{G}(df, \alpha)   + if \dot{d}_g^*\alpha + \dot{G}(\alpha, \alpha)  f\right) e_L.$$
   But each term is conformal to that of $g$ with conformal factor $e^{-\rho}$.
Hence
\begin{multline*}
\delta \nabla^* \nabla (f e_L)  \\
 =   - \half \rho \Delta f (x) -2 i \rho G( d f, \alpha) +
\left(i  \rho d_g^* \alpha + \rho G (\theta, \theta) \right) f e_L \\
= \rho \nabla^* \nabla (f e_L).
\end{multline*}

If $\nabla^* \nabla (f e_L)  u = -\lambda u$ then
\[
 \delta_g \nabla^* \nabla (f e_L) u(x)  =  - \lambda \rho u.
\]

To prove that the image of $D_2 \Phi_L$ is dense we argue by contradiction
and suppose that there exists $W \in L^2(X, L)$ such that
$$\int_X \delta_g \nabla^* \nabla h(f e_L, W(z))_h  dV_g = 0$$ for all $\rho$.
But this implies that $\int_X \rho (f e_L, W)_h dV_g = 0$ for all $\rho$,.
then $W = 0$. Write $W = F e_L$ so that the integral becomes,
$\int_X \rho f \bar{F} e^{-\psi} dV_g = 0$ for all $\rho$. This is only possible
if $ f \bar{F} e^{-\psi}  = 0$. But $f$ and $e^{-\psi}$ can only vanish on a set
of measure zero, so $F \equiv 0$ almost everywhere.

The image is closed because $\nabla^* \nabla$ is a Fredholm operator.
\end{proof}
\subsection{Varying the Hermitian connection}

In this section we fix $g, h$ and vary $\nabla \in \acal_h$. In the application to
Kaluza-Klein metrics, $P_h$ is fixed and the base and vertical metrics
are fixed and only the splitting into horizontal and vertical is varied.

We recall from Section \ref{GEOMSECT} that some of the variations are `trivial',
i.e., are within a gauge equivalence class. Bochner Laplacians with gauge-equivalent connections are unitary equivalent by a gauge transformation, i.e., they have the same spectrum and their eigensections are related by
a gauge transformations.  Viewed in terms of line bundles over $X$,
gauge equivalent connection forms are connection forms of a single connection
in two different unitary frames, hence  differ by a gauge transformation  $e^{i \theta} \in Map(X, S^1)$ taking $e_L \to  e^{i \theta} e_L$. The connection $1$-form then
changes by $id \theta \in \Omega^1(X, \R)$.  A unique representative
of a gauge equivalence class  is defined by the Coulomb gauge  $d^* a= 0.$

\begin{prop}

Suppose that     $(L, h, J) \to X$ is a
 Hermitian holomorphic  line bundle and let $\nabla \in \acal_h$ be given by
$\nabla (f e_L) =(d f + i f\alpha)\otimes e_L$ with $\alpha \in  \R$ in a unitary frame $e_L$.
Suppose that $L$ is non-flat, or if it is flat, that $d \alpha \not= 0$.
Then for generic gauge equivalence classes  $\alpha \in \acal_h \cong \Omega^1(X)$,  all of the eigenvalues of $\nabla_{g,h}^* \nabla_h$ are
simple and all of the eigensections have $0$ as a regular value. \end{prop}

\begin{proof}
Again by Proposition \ref{BOCHUNIT},
  $$  \nabla^* \nabla (f e_L)  = \left(- \Delta_g f   - 2 i  G(df, \alpha)   + if d_g^*\alpha + G(\alpha, \alpha)  f\right) e_L,$$
  where $\Delta_g f$ is the scalar Laplace operator. Taking the variation
  with respect to $\alpha$ gives,
    $$\delta  \nabla^* \nabla (f e_L)  = \left(   - 2 i  G(df, \dot{\alpha})   + if d_g^* \dot{\alpha} + 2G(\dot{\alpha}, \alpha)  f\right) e_L.$$


    If the image is not dense, there exists $W = F e_L$ so that
    $$\int_X \left(   - 2 i  G(df, \dot{\alpha})   + if d_g^* \dot{\alpha} + 2G(\dot{\alpha}, \alpha)  f\right) \bar{F} e^{-\psi} dV_g = 0,$$
    for all $\dot{\alpha} \in \Omega^1(X).$ We integrate $d_g^*$ by parts to get,
     $$\int_X \left(  ( - 2 i  G(df, \dot{\alpha})   + 2G(\dot{\alpha}, \alpha)  f) \bar{F}
     + i G(\dot{\alpha},  d (f \bar{F}))  \right) e^{-\psi} dV_g = 0.$$
         We may assume that the frame $e_L$ is unitary so that $\psi =0$.
     If $\beta\in \Omega^1(M, \C)$ and $\int_X G(\beta, \nu) d V_g = 0$ for all $\nu
     \in \Omega^1(M,\R)$, then $\beta= 0$. Indeed, we  may consider  $\nu$
     of the types $\nu = \nu_1 dx$, $\nu_2 dy$ separately to get orthogonality
     of the components $\beta_j$ with $\nu_j$. This reduces matters to the fact that
     if $u, v$ are complex-valued and $\int u v d V_g = 0$ for all $v$, then $u \equiv 0$.
     We conclude that
     $$   ( - 2 i  df   + 2 \alpha  f) \bar{F}
     + i   d (f \bar{F})   = 0 \iff ( -i df + 2 \alpha f)\bar{F} + i f d\bar{F} = 0. $$
     On any open set $U$ where  $f, F \not= 0$ we may divide by $if \bar{F}$ and write the solution as,
     $$\frac{d \bar{F}}{\bar{F}} = - ( - \frac{df }{f}+ 2 i \alpha ).$$
This implies that
     $$d \log \frac{\bar{F}}{f} = - 2 i \alpha \implies d \alpha = 0$$
     on a dense open set and since $\alpha \in C^{\infty}$, it is everywhere closed
     and hence the curvature of $(L,h)$ is zero.
    This is impossible unless $L$ is a topologically trivial line bundle, and the contradiction implies that $F \equiv 0$ except when $d \alpha = 0$.
 \end{proof}

\subsection{Proof of Theorem \ref{GENERIC1}}

 Eigenfunctions move continuously under perturbations of the operator. So it is easy to show that the set of metrics with for which the $j$th  eigenvalue is simple
is open. The difficulty is to prove that this set is dense.

To prove the first statement in Theorem \ref{GENERIC1} we need to verify the hypotheses of Theorem \ref{U} and therefore
need to prove Proposition \ref{PROP}, i.e., to determine the range $J$ of $D_2 \phi$.

To complete the proof of Theorem \ref{GENERIC1} it  suffices to prove:
\begin{prop} For each $m$,  $D_1 \alpha_m$ is surjective to $\C$.  \end{prop}

\begin{proof}

Let $G_{m, \lambda}(z,w)$ be the kernel of the Green's function
$G_{m, \lambda}: [\ker (D_m + \lambda)]^{\perp} \to  [\ker (D_m + \lambda)]^{\perp}$ for $D_m(g) + \lambda$
for a given background metric $g $. As above, one may use
the Hermitian metric $h$ on $K$ or the associated \kahler metric $g =\omega_J$ as the parameter space of metrics.

We  need to show that for each $x \in M$,
\[
\alpha_m: Q \times \{x\} \to \C: \alpha(u, \lambda, g, x) = u(x)
\]
has $0 \in \C$ as a regular value, i.e., that
$$D_1 \alpha (\cdot, x): T_{u, \lambda, b}(Q) \to \C,\;\; D_1 \alpha(\cdot, x)_{(u, \lambda, g)} (\delta u(x), 0, c, 0) = \delta u(x)$$
is surjective to $\C$, where $D_1$ is the differential along $Q$ with $x \in M$ held fixed. Since $x$ is fixed we may use a local coordinate $z$  and
frame $(dz)^m$ as above and identify local sections of $K^m$ with complex-valued functions  $u: U \to \C$, where $U$ is an open set containing $x$.

The constraint equation for $(v, 0, c, 0) \in T^*_{(u, \lambda, b)} Q$
 is
$$(D_m(g) + \lambda) v + (\dot{D}_m(g) + \dot{\lambda})u = 0,$$
and we can solve for $v \bot \ker(D_m(g) + \lambda)$ as
$$v(x) = - \int_M G_{m, \lambda} (x,y) \Pi_{\lambda}^{\perp} [  (\dot{D}_m + \dot{\lambda})u ](y) dV(y). $$
By Proposition \ref{SURJ2},  the range of $D_2 \Phi$, i.e., the set of functions $ [  (\dot{D}_m + \dot{\lambda})u ] $, spans $L^2_0$. Therefore,
the image  $\Pi_{\lambda}^{\perp} [  (\dot{D}_m + \dot{\lambda})u ]$   spans
$[\ker (D_m + \lambda)]^{\perp}$. It follows that the possible values
of $v$ are all functions of the form,
\begin{equation*}
v(x) = \int_M G_{m, \lambda} (x,y) f(y) dV(y),
\end{equation*}
where  $f \bot \ker (D_m(g) + \lambda)$. Thus,  $D_1 \alpha $
is  surjective to $\C$ unless for all $j \bot \ker (D_m(g) + \lambda)$,
either the real or imaginary parts of
$$G_{m, \lambda} (j)(x) =\int_M G_{m, \lambda}(x,y) j(y) dV(y) $$
vanish (or both) for every such $j$.

Since $j = [D_m(g) + \lambda] f$ where $\int f = 0$ we would  get  the absurd conclusion that
$$f(x) = 0, \;\; \forall f \bot \ker (D_m(g) + \lambda). $$
Equivalently,
$$G_{m, \lambda} (x,y) +  u_{\lambda}(x)  = 0.  $$
This is not possible and the contradiction ends the proof.
\end{proof}

\subsection{Multiplicity of the spectrum of $\Delta_g$}\label{simple}
We begin by observing that $\lambda_{m,j} = \lambda_{-m,j}$ and that $\phi_{-m,j} = \overline{\phi_{m,j}}$. Then any real eigenfunction which is a linear combination of $\phi_{m,j}$ and $\phi_{-m,j}$ is
\[
\Re \left(e^{i\theta_0}\phi_{m,j}\right)
\]
for some constant $\theta_0$. In local coordinates, $\phi_{m,j}$ is $\tilde{\phi}(z)e^{im\theta}$, and therefore we see that
\[
\Re \left(e^{i\theta_0}\phi_{m,j}\right) = \Re \left(T_{\theta_0}\phi_{m,j}\right) = T_{\theta_0}\Re \left(\phi_{m,j}\right).
\]
For $m_1$ and $m_2$ such that $|m_1| \neq |m_2|$, we argue that $\lambda_{m_1,j_1} \neq \lambda_{m_2,j_2}$ is satisfied for an open dense subset of metric $G$. This immediately implies the first, third, and the fourth statement of Theorem \ref{MAIN}. Note that the eigenvalue moves continuously with respect to $G$. So it is sufficient to prove that
\begin{lemm}\label{simplesimple}
Let $P \to X$ be a non-trivial principal $S^1$ bundle. Fix integers $m_1$ and $m_2$ such that $|m_1| \neq |m_2|$.  Among all $S^1$-invariant metric $G$ on $P$, $G$ satisfying $\lambda_{m_1,j_1} \neq \lambda_{m_2,j_2}$ is dense.
\end{lemm}
\begin{proof}
The deformation of the base of the Kaluza-Klein  metric
  does not touch the vertical operator $\frac{\partial}{\partial \theta}$
  and therefore the first order perturbation equations for infinitesimal
  deformations of the base metric $g$ gives,
    \begin{equation} \label{EIG4}
      (\dot{\Delta}_H +   \dot{\lambda}_{m,j}) \phi_{m,j} =  (\Delta_H+\lambda_{m,j} + m^2)  \dot{\phi}_{m,j}
    \end{equation}
   Taking the inner product  with $\phi_{m,j}$  gives
    \begin{equation} \label{EIG5}
      -  \dot{\lambda}_{m,j} =  \langle \dot{\Delta}_H  \phi_{m,j},  \phi_{m,j}\rangle.
   \end{equation}
If there exist weights $m_1 \neq m_2$ for which we cannot split the eigenvalue
$\lambda_{m_1,j_1} = \lambda_{m_2, j_2}$ then for all infinitesimal base perturbations
$\rho$ we get
   \begin{equation} \label{EIG6} \begin{array}{l}    \dot{\lambda}_{m_1,j_1} =   \dot{\lambda}_{m_2,j_2} \\ \\
   \iff \langle \dot{\Delta}_H  \phi_{m_1,j_1},  \phi_{m_1,j_1}\rangle = \langle \dot{\Delta}_H \phi_{m_2,j_2}, \phi_{m_2,j_2} \rangle.
   \end{array} \end{equation}
   Write $\phi_{m,j} = f_{m,j} (dz)^m$. Differentiation of the eigenvalue equation therefore gives the well-known formula
    $$\langle \dot{D}_{m_1} f_{m_1, j_1}, f_{m_1, j_1} \rangle
    = \langle \dot{D}_{m_2} g_{m_2, j_2}, g_{m_2, j_2} \rangle$$
    for every variation of $g$, where the inner product is that of $g_0$.

    Recall from previous section that $\dot{\Delta}_H = \rho \Delta_H$. Because $-\Delta_H \phi_{m,j} = (\lambda_{m,j} - m^2) \phi_{m,j}$ we have for any $\rho \in C^{\infty}(X)$,
\begin{multline*}
(\lambda_{m_1,j_1} - m_1^2)  \int_X \rho | f_{m_1, j_1}|^2 e^{-m_1 \phi} dA_0 \\
= (\lambda_{m_2,j_2} - m_2^2)  \int_X \rho | f_{m_2, j_2}|^2 e^{-m_2 \phi} dA_0.
\end{multline*}
   Thus,
   $$(\lambda_{m_1,j_1} - m_1^2)   | f_{m_1, j_1}|^2 e^{-m_1 \phi}= (\lambda_{m_2,j_2} - m_2^2)   | f_{m_2, j_2}|^2 e^{-m_2 \phi}. $$

   Integrating both sides against $dV_g$ and using that both eigenfunctions
   are $L^2$ normalized gives
      $$(\lambda_{m_1,j_1} - m_1^2) =    (\lambda_{m_2,j_2} - m_2^2),$$
  i.e., $|m_1|=|m_2|$.
\end{proof}

\section{\label{LOCALSECT} Local structure of eigensections at zeros }

To study the nodal sets of real and imaginary parts of Kaluza-Klein Laplacians, we first
study the zeros of the associated sections of the line bundles. For simplicity of exposition,
we assume that $L = K$ and describe the zero sets of eigen-$m$-differentials. Essentially
the same discussion is valid for other line bundles.

We follow the notation and terminology in the theory of holomorphic quadratic differentials,
even though our eigendifferentials are $C^{\infty}$, usually not holomorphic and of
general weight $m$.
Following a standard terminology for quadratic differentials, we call
a point $z$ such that $f_{m,r_j}(z)\not= 0$ a ``regular point'' and
a point where $f_{m,r_j}(z) = 0$ a ``critical point'' or a ``singular point''.

 After the first version of this article was written, we  located some recent articles generalizing the geometric properties of quadratic
 differentials on Riemann surfaces to $C^{\infty}$  higher order differentials \cite{FN12, AM17}
 and to other line bundles. We now use the terminology and results of these articles but have retained some
 from our first version since it is important for us to lift to $P_h$.

\subsection{Trajectories of eigen-differentials. }

The real and imaginary parts of the  eigendifferentials $\omega_{m,j} =f_{m,j}(z) (dz)^m$ are called binary differentials of degree $m$
and the equation for the  zero set of $\Im \omega_{m,j}$  is called a  binary differential equation of degree $m$ \cite{FN12}. It is traditional to consider the nodal set $\Im f_{m,j}(z) (dz)^m = 0$. If there exist exactly $m$
solutions at a regular point where $\omega_{m,j}(z) \not=0$ then $\omega_{m,j}$ is called {\it totally real} in \cite{FN12}.
Our $m$-differentials are of a special type since they are real and imaginary parts of $f_{m,j}(z) (dz)^m$ and therefore only
have terms of the form $(dz)^m$ or $(d\bar{z})^m$.
The following is the key input into Proposition \ref{COVER}.

\begin{lemm} \label{REG} $\Im f_{m,j} (dz)^m$ is a totally real $m$-differential.  At a regular point $z$, there exist
 $m$ distinct  solutions $v$ of  $\Im f_{m,j} (dz)^m(v)= 0$ in $T_z X$.

 \end{lemm}

  \begin{proof}
It f $v = (\cos \phi, \sin \phi)$,
then in the notation of \eqref{ubab}, the equation is
$$ \left(a_{m,j} \mathfrak{c}_m -  b_{m,j} \mathfrak{s}_m\right)(\cos \theta, \sin \theta) = 0. $$
Here $\mathfrak{c}_m = \Re (\cos \theta + i \sin \theta)^m = \cos m \theta,$ and the equation is
\[
a_{m,j}(z) \cos m \theta -  b_{m,j}(z) \sin m \theta = 0 \iff \tan m \theta = \frac{a_{m,j}}{b_{m,j}},
\]
where $a_{m,j}, b_{m,j} \in \R$ and where we assume with no loss of generality that $b_{m,j} \not= 0$.
Since the principal branch of $\tan^{-1}: \R \to (-\pi/2, \pi/2)$ is one-to-one, there exists precisely one solution
$\theta_0$ of $\tan m \theta = \frac{a_{m,j}}{b_{m,j}}$ with $m \theta \in (-\pi/2, \pi/2)$, namely the principal branch of
$\tan^{-1} (\frac{1}{m} \frac{a_{m,j}}{b_{m,j}}).$ Since $\tan \theta$ is $\pi$-periodic,  $\tan m \theta$ is $\frac{\pi}{m}$-periodic, and the full set
of solutions is $\theta_0 + k \frac{ \pi}{m}$ with $k = 0, \dots, m-1$.
\end{proof}

 The kernel of $\Im f_{m,j} (dz)^m$ defines a  smooth $m$-valued distribution on $X$ with singularities where $\omega_{m,j} = f_{m,j} (dz)^m= 0$.
The $m$ line fields defines a {\it web} of m transverse singular foliations, whose leaves are called  the {\it  trajectories}:

\begin{defi}

The trajectories of the $m$ differential $f_{m,j}(dz)^m$ are the  integral curves of the  kernel of $\Im f_{m,j} (dz)^m$, i.e.
the trajectories are the (smooth) curves $\gamma(t) $  in $X$ along which $\Im \phi_{m,j}(\gamma(t), \gamma'(t)) = 0$.  \end{defi}

\begin{rema}
A trajectory in this sense of this article is called a `horizontal trajectory' in \cite[Definition 5.5.3]{Str}.
They are illustrated in
\cite[Section 7]{Str} for holomorphic quadratic differentials. Illustrations of webs for higher order real differentials can
be found in \cite{FN12}.

\end{rema}

Trajectories downstairs on $X$  lift to $P_h$ by their tangent vectors.
  A trajectory $\gamma_{z_0, \theta_0}(t)$
downstairs is a smooth curve along which  $$\Im (\phi_{m,j} (\gamma_{z_0, \theta_0}(t), \dot{\gamma}_{z_0, \theta_0}(t)) = 0.$$
It lifts to a smooth curve $ (\gamma_{z_0, \theta_0}(t), \dot{\gamma}_{z_0, \theta_0}(t))$  in the nodal set upstairs.   Since $d \pi$ is an isomorphism, the trajectories
are special curves on the nodal set $\Im \phi_{m,j} = 0$.

\subsection{Non-degenerate singular points}

The structure of the trajectories through a singular (zero) may be complicated in general if no conditions
are placed on the degeneracy of the zeros. The  purpose
of Theorem \ref{GENERIC1} is to allow us to assume that the zeros are of first order, so that they are
isolated and non-degenerate.

The structure of the trajectories of a totally real $m$-differential near an isolated singular point is discussed  in \cite{FN12}.
As with vector fields, the key topological invariant of the singular point is its {\it index}


\begin{defi} The index of a singular point $z_0$ where
\[
f_{m,j}(z_0) (dz)^m = 0
\]
is related to the  degree of the circle map defined
by
\[
\delta(t) = z_0 + r e^{it} \to \frac{f_{m,j}(\delta(t))}{|f_{m,j}(\delta(t))|}
\]
on a small circle around $z_0$ to $S^1$ by
\[
\mathrm{ind}(z_0) = \frac{\pm 1}{m}  \mathrm{deg} \frac{f_{m,j}(\delta(t))}{|f_{m,j}(\delta(t))|}.
\]
Equivalently, in a small circle $C$ around $z_0$,  choose a unit vector $X(0) \in \ker \Im f_{m,j} (dz)^m |_{C(0)}$
where $C(t): [0,2\pi] \to X$ is a constant speed  parametrization of $C$ and let $\ell = L(C)$ be its length.  Let  $X(t)$ be a smooth extension of $X(0)$ along $C(t)$.
After a complete turn, $X(2\pi)$ must be one of the $2m$ solutions of $\omega(X) =0$. After $2m$ turns $X(2m \ell) = 0.$ Let $\theta(t)$
be a smooth determination of the angle between the tangent line to $C$ and $X(t)$. Then $\theta(2m \ell)$ and $\theta(0)$ differ
by an integer multiple of $2 \pi$. The index of $z_0$ is defined by
$$\mathrm{ind} (\omega, z_0) = \frac{\theta(2 \pi m \ell) - \theta(0)}{4 \pi m}. $$
 \end{defi}
 Thus, the index has the form $\frac{s}{2m}$ with $s \in \Z$.
The following Lemma shows that singular points must exist when the genus of $X$ is non-zero.

\begin{lemm} If $f_{m,j}(dz)^m$ has isolated non-degenerate  zeros, then the sum of the indices of the zeros is
 the Chern class of $K_X^m$. \end{lemm}




\begin{lemm}  If $z_0$ is a non-degenerate singular point (zero of order $1$) of $f_{m,j}(dz)^m$,
then $\mathrm{ind}(\omega_{m,j}, z_0) = \frac{\pm 1}{m}$.  \end{lemm}

\begin{proof} This follows from the fact that $f_{m,j}$ is linear in this case and hence the degree of the
associated circle map is $\pm 1$.\end{proof}

\begin{prop} \label{GENINDEX} For a generic Riemannian metric $g$ on $X$, all singular points of all eigendifferentials of $\nabla^*\nabla$
on $K^m$ have index $\frac{\pm 1}{m}$ for all $m \not= 0$. \end{prop}

\begin{proof} It  is part of  Theorem \ref{GENERIC1}, all singular points are non-degenerate.
To prove this it suffices to show that the coefficients $f_{m,j}$ are linear near each singular point. This follows
from the   Bers local formula for eigensections around a zero. We use Proposition \ref{ddbarprop} to Taylor expand the operator
 $$\begin{array}{lll} D_m =  \nabla_m^* \nabla_m &=&  2 g^{1 \bar{1}} \frac{\partial^2}{\partial z \partial \bar{z}}  f - 2 m [\frac{\partial f}{\partial \bar{z}} g^{1 \bar{1}}]  \frac{\partial \phi}{\partial z} + K f, \end{array} $$
 around a nodal point.

Let $p$ be a nodal point of $f_{m,j}$.   We  Taylor expand the coefficients in \kahler normal coordinates for $(J, g)$
 in a disc $z \in D(p, r)$  to get
 \begin{itemize}
 \item $g^{1 \bar{1}} = 1 + K(p) |z|^2 + \cdots$; \bigskip

 \item $ \frac{\partial \phi}{\partial z} =  \frac{\partial \phi}{\partial z}(p) + \omega_p \bar{z} + \cdots = \bar{z}+ \cdots  . $

 \end{itemize}
 Thus, the {\it osculating} constant coefficient operator is
 $$D_m^p f = 2 \frac{\partial^2}{\partial z \partial \bar{z}}  f
  + K(p) f. $$

    Let $\pcal_k$ denote homogeneous polynomials of degree $k$ in $z, \bar{z}$.  It is better to arrange the terms of the Taylor expansion of $D_m$ at $p$ into
 terms
 $$D_m = L_{-2} + L_{-1} + L_0 + L_1 + \cdots \cdots $$
 where $L_j : \pcal_k \to \pcal_{k + j}.$

 Thus, $L_{-2} = \frac{\partial^2}{\partial z \bar{z}}$,
 $L_{-1} = 0$, $L_0 = K_{[2]} \frac{\partial^2}{\partial z \bar{z}}   - 2 m [\frac{\partial f}{\partial \bar{z}}] \bar{z} + K(p)$ etc. Note that $L_{-1} = 0$ because $d g^{1 \bar{1}}(p) = 0$ and $\partial \phi(p) = 0$, so neither
 the second or first derivative terms contribute at this order.

 Also expand
 \begin{multline*}
 f(z) =\\
  f_1(p) z  + f_{\bar{1}}(p) \bar{z} + f_{11} (p) z^2 + f_{1 \bar{1}}(p) |z|^2 + f_{\bar{1} \bar{1}} (\bar{z})^2 + \cdots + f_{[k]} + \cdots,
  \end{multline*}
 where $f_{[k]} \in \pcal_k$ is homogeneous of order $k$.

The following is the generalization of the Bers local expansion theorem
to complex line bundles.
\begin{lemm} Let $z_0$ be a zero of $f_{m,j} (dz)^m$. The first non-zero homogeneous term $f_{[n]}$ of the Taylor expansion of an eigenfunction is a harmonic homogeneous polynomial. If the order of vanishing is $n$, $f_{[n]}(z) = a \Re z^n + i b \Im z^n$.
In particular, at a non-degenerate zero, the first homogeneous term is $f_{m,j} = a \Re z +i b \Im z.$
\end{lemm}

\begin{proof}

 It is evident that
 $L_{-2} =  \frac{\partial^2}{\partial z \partial \bar{z}}: \pcal_k \to \pcal_{k-2}$. If $f_{[k]}$ is the term of lowest
 degree in the expansion of $f$ then $\frac{\partial^2}{\partial z \partial \bar{z}} f_{[k]}= 0$, i.e., $f_{[k]} $ is a homogeneous harmonic polynomial.
 In real dimension $2$ the only possibilities are linear combinations of the  real and imaginary
 parts of $z^k$. By a well-known argument, the nodal set of the real and
 imaginary parts of $f$ are topologically equivalent to those of the leading
 order homogeneous term.


 \end{proof}

This completes the proof of the Proposition.

 \end{proof}

\section{Adapted Kaluza-Klein metrics}

All of the Kaluza-Klein  metrics are Riemannian metrics on principal $S^1$ bundles
$P_h \to X$ associated to $C^{\infty} $ complex Hermitian line bundles $L \to X$.  Given $P_h$ we recover $L$ as an associated line bundle.
 Let $(X,g)$ be any Riemannian
surface. We denote the genus of $X$ by $g$.  Let $(L, h) \to X$ be any complex line bundle with Hermitian metric $h$. Associated to $L$
is the $U(1)$  bundle $P_h$ of orthonormal frames.
 Let $T= \frac{\partial}{\partial \theta} $
generate the $S^1 \cong U(1)$ action. We endow $P_h$
with a connection $\alpha$, that is, an $S^1$ invariant $1$-form on $P_h$ such that $\alpha(T)=1$.

The connection defines a splitting $$T_p P_h = H_p \oplus V_p$$
into horizontal and vertical spaces. The vertical space is given by orbits
of the $S^1$ action. The horizontal space is defined by $H_p = \ker \alpha$ and is
isomorphic under $d\pi_p$ to $T_z X$ where $\pi(p) = z$.

\begin{defi} \label{SASAKIDEF2}The Kaluza-Klein metric on
 $P_h$ is the  $S^1$-invariant metric $G$ such that the horizontal
space $H_p  : = \ker d\alpha$ is isometric to $T_{\pi(p)} X$, so that $V = \R \frac{\partial}{\partial \theta}$ is
orthogonal to $H$ and is invariant under the natural $S^1$ action and so that the fiber  is a unit speed geodesic. \end{defi}

A Kaluza-Klein  metric on the principal $S^1$ bundle $P_h$  is thus determined by the pair $(g,\alpha)$ where
$g$ is a metric on $X$ and where $\alpha$ is a connection $1$-form on $F$. In general, the metric
and connection are chosen independently. In Section \ref{SPECIALC} we discuss the orthonormal
frame bundle,  where
$\alpha$ is the Riemannian connection of  $g$.

Given $P_h$ and any character $\chi_m =e^{im \theta}$ of $S^1$ we obtain associated line bundles (resp, real rank 2
bundles) by
$$L^m = P_h \times_{\chi_m} \C. $$
For purposes of this paper it may be assumed that $m \geq 0$.

We often assume
that $X$ is equipped with a complex structure $J$ and that $L$ is a holomorphic line bundle.
Let $D_h^* \subset L^*$ be the unit co-disc bundle with respect to $h$ and let
$P_h = \partial D_h^* $ be its boundary, an $S^1$ bundle $\pi: P_h \to X$. \bigskip

\begin{rema}
Not all $S^1$ invariant metrics on  $P_h$ are adapted Kaluza-Klein  metrics. It would be interesting to consider
more general $S^1$-invariant metrics on $SX$ or  on other  manifolds (of all dimensions),
as well as invariant metrics under more general compact Lie groups. \end{rema}

\subsection{\label{ADAPTEDSECT} Geometry and analysis of  Kaluza-Klein metrics}

We use the term Kaluza-Klein metric or Kaluza-Klein metric in the sense of Definition \ref{SASAKIDEF2}
to denote metrics $G$ on the unit tangent bundles $\pi: P_h = S X \to X$ over surfaces for which
the vertical and horizontal spaces are orthogonal and which is invariant under the  free  $S^1= SO(2)$
action \footnote{A free action is one for which all isotropy groups are trivial.}.  They
are special cases of Riemannian submersions with totally geodesic fibers isometric to
$\R/2 \pi \Z$.

\begin{defi}\label{ADAPTEDDEF}
If $H$ is a compact Lie group, and if  $\pi: M \to B$ is a  principal $H$-bundle with fiber $F$,
then one says that a connection $\theta$ is an $H$-connection for $\pi$ if the $H$ action preserves
the horizontal spaces and preserves the connection $1$-form. One says that a metric $G$
is {\it adapted} to $H$ if the fibers $\pi^{-1}(b)$ are totally geodesic and isometric to $F$ and
such that the horizontal distribution of $\theta$ is the orthogonal complement to the vertical.
\end{defi}

The following Lemma
gives details on the equivalences of the various conditions and  is implicitly contained
in \cite[Example 2.1]{BBB} and is proved in  \cite[Theorem 3.5]{V}. Hence we only sketch
the proof.
\begin{lemm} Suppose that $S^1$ acts freely on $M$ and that $G$ is an  $S^1$ invariant metric
for which all orbits are geodesics isometric to $\R/2 \pi \Z$. Then $G$ is a Kaluza-Klein metric and $\pi: M \to M/S^1$ is a Riemannian submersion with totally geodesic fibers
isometric to $\R/2 \pi \Z$. \end{lemm}

\begin{proof}  Under the freeness assumption, we have an $S^1$ bundle $\pi: M \to X : = M/S^1$.
Let $V = \R \frac{\partial}{\partial \theta}$ be the vertical space, i.e., the  tangent space to the orbits.
Let $H_x = V_x^{\perp}$.
The metric $G$ determines a quotient Riemannian metric on $X$ using the isomorphisms
$d\pi_x:  H_x \to T_{\pi(x)} X. $  By assumption, $|\frac{\partial}{\partial \theta}|_G = 1$ if we identify
$S^1 = \R /(2 \pi \Z$. The only non-trivial statement is that the orbits are geodesics. Let $\alpha(t)$
denote the orbit of a point $x$ under the $S^1$ action. Let $v(0)$ be
a horizontal vector at $\alpha(0)$. Let $\sigma(t,s)$ be the parallel translation of $\alpha(t)$ along
a curve in $X$ with initial tangent vector $d \pi(v(0))$. Then $\frac{\partial}{\partial s} \sigma(t, 0)$
is a horizontal vector field along $\alpha(t)$. The arclength of the curve $t \to \sigma(s, t)$ is constant
in $s$.  The first variation formula implies that $\alpha(t)$ is geodesic.

\end{proof}

These adapted metrics are special cases of invariant metrics on an $S^1$-manifold $M$. For instance, in the case of the  non-free action of $S^1$ on the standard  $S^2$ by rotations around the $x_3$-axis,
$|\frac{\partial}{\partial \theta}|_{g_0}$ varies with the orbit, and almost no orbits are geodesics. It would
be interesting to consider generalizations to all $S^1$ invariant metrics.

Sections of $K^m$, i.e., differentials of type $(dz)^m$, lift to the
dual line bundle $K^* = T^{(1,0)}$ as equivariant scalar functions
$F: K^* \to \C$ transforming by $e^{im \theta}$ under the $S^1$ action
of rotating a frame. Using the metric $g$ we form the unit tangent
bundle $\pi: SX \to X$.  In this section, we review the relevant formulae for
lifted operators and review the fact that the Bochner Laplacians are
Fourier components of the horizontal Laplacian on $SX$.

\subsection{\label{LIFTSECT} Lifts to $P_h$}
The natural inner product on $L^2(P_h, dV_G)$ is given by
$$\langle f, f \rangle = \int_{P_h} |f|^2 dV_G.$$

Sections $s$ of $L^m$ naturally lift to $L^*$ and $F_h$ by
$$\hat{s}(z, \lambda):= \lambda(s(z)).$$
It is straightforward to check that the lift of $s \in C(X, L^m)$ satisfies
$\hat{s}(r_{\theta} x) = e^{im \theta} \hat{s}(x)$ and that
$$\int_{F_h} |\hat{s}(x)|^2 dV_G = \int_X \|s(z)\|_{h^m}^2 dA_g.$$
Indeed, if $x = r_{\theta} \frac{e_{L^*}(z)}{\|e_{L^*}(z)\|}$ then
$\hat{s}(x) = e^{i m \theta}\|e_L(z)\|_{h^m}^{m} $.

In the case of $L= K_X$,
the lift has the form,
\begin{equation*} \widehat{f (dz)^m} (Y) = f (dz(Y))^m. \end{equation*}
We define a orthonormal frame of $T^*X$ by $\omega_1 = e^{-\phi} dz: =  \frac{dz}{|dz|_h}$ as above, and let $\|\frac{\partial}{\partial z} \|^{-1} \frac{\partial}{\partial z}= e^{\phi} \frac{\partial}{\partial z}$ be the dual frame. In local coordinates
$z,\bar{z}$ on $X$ and in this local frame we define local coordinates $(z, \bar{z}, \theta)$ on $S X$ corresponding
to the point $e^{i \theta}   e^{\phi} \frac{\partial}{\partial z}.$

Then $(dz)^m$ lifts to the function,
$$e_m(z, \bar{z}, \theta) = \widehat{(dz)^m}(e^{i \theta}   e^{\phi} \frac{\partial}{\partial z}) = e^{i m\theta} e^{m \phi(z)}. $$
Consequently,  the eigendifferential $f_{m,j} (dz)^m$ lifts to
$$\phi_{m,j}(z, \bar{z}, \theta) =  f_{m,j}(z) e^{i m\theta} e^{m \phi(z)}. $$
In \eqref{uv} we decomposed the lift into real and imaginary parts. We
now relate them to the real and imaginary parts of $f_{m,j}$.

If we take the inner product of $u_{m,j}$ and $ v_{m,j}$ just along the fiber
and use orthogonality of $\cos m \theta, \sin m \theta$ and that $\int_0^{2 \pi}  (\cos^2 m \theta - \sin^2 m \theta) d \theta = 0$, and then integrate in $dA(z)$
we get

\begin{lemm}
$\begin{array}{l} \langle u_{m,j}, v_{m,j} \rangle = 0.\end{array}. $
\end{lemm}

\subsection{Eigenspace decompositions}

 The Kaluza-Klein Laplacian has the form
\begin{equation*}
\Delta_G = \Delta_H +  \frac{\partial^2}{\partial \theta^2},\;\;\;  \mathrm{where}\; \Delta_H = \xi_1^2  + \xi_2^2
\end{equation*}
is the horizontal Laplacian. The fact that the fiber Laplacian is $ \frac{\partial^2}{\partial \theta^2}$ reflects
the fact that $S^1$ orbits are geodesics isometric to $\R/2 \pi \Z$.

The weight spaces
are $\Delta_H$-invariant, i.e., as an unbounded self-adjoint operator,
$$\Delta_H: \hcal_m \to \hcal_m. $$
Under the  canonical identification
$$\hcal_m \cong L^2(X, L^m) $$
using the lifting map and $\Delta_H |_{\hcal_m} \cong D_m - m^2 I$ under the lifting map.

We then consider joint eigenfunctions $\phi_{m, j}$ of the Kaluza-Klein Laplacian $\Delta_G$ and of $\frac{\partial}{\partial \theta}$.
 The commutation relations show that $[\Delta_G,  \frac{\partial^2}{\partial \theta^2}] = 0.$

\begin{lemm}\label{BOCHNERLEM} The Bochner Laplacian agrees with the horizontal Laplacian $\Delta_H$. In the above local coordinates and frame,
$$ \widehat{\nabla_m^* \nabla_m (f (dz)^m)} =
\Delta_H   \widehat{(f (dz)^m)}.$$
\end{lemm}


Note that  except for the last identity, these statements are true for any isometric $S^1$ action, not just for adapted Kaluza-Klein metrics.

\subsection{Equivariant decomposition}

Since $S^1$ acts isometrically on $(M, G)$ we may decompose into its weight spaces,
\begin{equation*}
L^2(M, dV_G) = \bigoplus_{m \in \Z} \hcal_m,
\end{equation*}
where $$\hcal_m = \{f: M \to \C: f(e^{i \theta} x) = e^{i m \theta} f(x)\}. $$ The weight spaces
are $\Delta_H$-invariant, i.e., as an unbounded self-adjoint operator,
$$\Delta_H: \hcal_m \to \hcal_m. $$
The lifting map gives  a canonical identification
$$\hcal_m \cong L^2(X, L^m). $$

\section{Connectivity of nodal sets of Kaluza-Klein eigenfunctions}


Given the preparations in Section \ref{LOCALSECT}, it is now a simple matter to prove Theorem \ref{NODALCON}.
The following is an immediate consequence of Lemma \ref{REG}:

\begin{lemm} \label{COVERLEM}  If $0$ is a regular value, then $\ncal_{u_{m,j}} \subset S X$ is a singular $2m$-fold cover of $X$ with blow-down singularities over  points where $f_{m,j}(z) (dz)^m = 0$.
\end{lemm}

Indeed, the $2m$ zeros of $\Im \omega_{m,j}(v) = 0$ in $S_z X$ give $2m$ points on the fiber $\pi^{-1}(z)$ in $P_h$. Since
locally there exist $2m$ smooth determinations of the zeros, the nodal set is a covering map away from the singular points.

We have separated Lemma \ref{COVERLEM} from further geometric results on the map  $\pi : \ncal_{u_{m,j}}  \to X$  in the next section since it
was stated separately from those results in Theorem \ref{NODALCON}.

\begin{rema} In the literature,  $\pi : \ncal_{u_{m,j}}  \to X$  is sometimes
called a branched cover \cite[Section 4]{AM17}, but as J. Y. Welschinger explained to us, the terminology is misleading since smooth branched covers are supposed to
have $z^{1/m}$ singularities over the branch points,  just as for holomorphic branched covers, while the inverse image of a zero of $f_{m,j}$ is an $S^1$ orbit and the singularity is blown up. In some sense, $\pi$ is locally like the projection of a vertical helicoid onto the horizontal plane.

\end{rema}





\section{\label{NODALDOMAINSECT1} Nodal domains of real and imaginary parts}

We now give a sketch of the proof of Theorem \ref{MAIN}.

 By Proposition \ref{COVER} (and \eqref{SIGMADEF}), $\ncal_{\Im \phi_{m,j}}   \backslash (\ncal_{\Im \phi_{m,j}} \cap \Sigma) \to X \backslash \zcal_{f_{m,j}}$
is a $2m$-sheeted cover. Moreover, $P_h \backslash \Sigma \to X$ is an $S^1$ bundle and
\begin{equation*}i
(P_h \backslash \Sigma ) \backslash \ncal_{\Im \phi_{m,j}} \to X
\end{equation*}
is a fiber bundle whose fibers consist of the punctured fibers $\pi^{-1}(z) \backslash \ncal_{\Im \phi_{m,j}}$. The connected components
of each punctured fiber consist of `arcs' along which $\Im \phi_{m,j}$ has a constant sign. We therefore express it as
\begin{equation*}
(P_h \backslash \Sigma ) \backslash \ncal_{\Im \phi_{m,j}} = \pcal_+ \bigcup \pcal_-
\end{equation*}
where $\mathrm{sign} \Im \phi_{m,j} = \pm$ in $\pcal_{\pm}$. Each $\pi: \pcal_{\pm} \to X$ is a fiber bundle whose fiber consists
of $m$ arcs of the fibers of $\pi: P_h \to X$.  Since the number of zeros in each regular fiber is $2m$, the
number of connected components of $\pcal_{\pm}$ is $\leq m$. When we take the closure of these sets (i.e., add in the
singular fibers, on which $\Im \phi_{m.j} = 0$, the connected components of the closure are the nodal domains. It follows
that there are $\leq 2m$ nodal domains. We now argue that the closure of $\pcal_{\pm}$ is connected, so that there
exist exactly $2$ nodal domains.

We now use the local analysis in Section \ref{LOCALSECT} of eigendifferentials of generic Bochner Laplacians around their zeros to determine
how the sheets are connected at the singular fibers $\ccal_j = \pi^{-1}(z_j)$,
corresponding to singular points (i.e., zeros) of $f_{m,j}(dz)^m$ i.e., we consider the maximal components $\pcal_{\pm, j}$ of
$$  \pcal_{\pm}  \backslash \bigcup_{j =1}^m \ccal_j = \bigcup_{j=1}^m \pcal_{\pm, j}, $$
in which $\Im \phi_{m,j}$ has a single sign. When we union the left side with $ \bigcup_{j =1}^m \ccal_j $
we glue together some of these domains along intervals of the singular fibers.

 The gluing rule for the nodal domains  is determined by the gluing rule for the nodal set, since the boundary of the
 the each nodal domain is the nodal set. From the downstairs point of view, the gluing rule is the monodromy of
 the cover  $\ncal_{u_{jm,j}} \to X \backslash Z(\omega_{m,j}) $ If we fix a singular point $z_0$, then we get a monodromy representation
$$\rho: \pi_1( X \backslash Z(\omega_{m,j})) \to Aut(\pi^{-1}(z_0)), $$
determining how  the sheets of the nodal set are changed as the point circles around $z_0$.

By Proposition \ref{GENINDEX}, the index of the singular points $z_0$ is $\frac{\pm 1}{m}$. In terms of the monodromy, this
means precisely that each turn around a circle $C$ enclosing $z_0$ lifts to an arc from one vector in the fiber to its nearest
neighbor with the same sign of $\Re \phi_{m,j}$ (i.e., skipping the neighboring vector of the opposite sign).

 It follows
  that both  the $+$ region and $-$ region  is connected in $P_h$. Hence there are just two nodal domains.

\subsection{ Counting the number of nodal domains}
We now give a more detailed presentation.

Let $D$ be an open disc. We first study connectivity of a certain graph that arise from a pair of partitions of $D$.

Let $P$ and $Q$ be partitions of $D$, i.e., $P$ (resp. $Q$) is a collection of disjoint open-sets $\Omega_P(1),\ldots, \Omega_P(n_P)\subseteq D$ (resp. $\Omega_Q(1),\ldots, \Omega_Q(n_Q)\subseteq D$) such that
\[
\overline{\cup_{k=1}^{n_P} \Omega_P(k)} =D ~(\text{resp. }\overline{\cup_{k=1}^{n_Q} \Omega_Q(k)} =D).
\]
Let $c_P: P \to \{0,1\}$ and $c_Q: P \to \{0,1\}$ be colorings of $P$ and $Q$, and define the inversions of $c_P$ and $c_Q$ by $c_P' = 1-c_P$ and $c_Q' = 1-c_Q$.

We now define a graph $G_m(P,Q,c_P,c_Q)$ as follows:

The vertex set is
\begin{align*}
\left\{\begin{array}{cccc}
v_{1,1}, & v_{1,2}, & \cdots & v_{1,n_P},\\
v_{2,1}, & v_{2,2}, & \cdots & v_{2,n_Q},\\
v_{3,1}, & v_{3,2}, & \cdots & v_{3,n_P},\\
v_{4,1}, & v_{4,2}, & \cdots & v_{4,n_Q},\\
&\vdots & &\\
v_{4m,1}, & v_{4m,2}, & \cdots & v_{4m,n_Q}
\end{array}\right\}
\end{align*}
and edges are
\begin{align*}
\{v_{4j, a},v_{4j+1, b}\} &\text{ such that } \Omega_Q(a)\cap \Omega_P(b) \neq \emptyset, \text{ and } c_Q'(\Omega_Q(a)) = c_P(\Omega_P(b)),\\
\{v_{4j+1, a},v_{4j+2, b}\}&\text{ such that } \Omega_P(a)\cap \Omega_Q(b) \neq \emptyset, \text{ and } c_P(\Omega_P(a)) = c_Q(\Omega_Q(b)),\\
\{v_{4j+2, a},v_{4j+3, b}\} &\text{ such that } \Omega_Q(a)\cap \Omega_P(b) \neq \emptyset, \text{ and } c_Q(\Omega_Q(a)) = c_P'(\Omega_P(b)),\text{ and}\\
\{v_{4j+3, a},v_{4j+4, b}\} &\text{ such that } \Omega_P(a)\cap \Omega_Q(b) \neq \emptyset, \text{ and } c_P'(\Omega_P(a)) = c_Q'(\Omega_Q(b))
\end{align*}
for $j=0,1,\ldots, m-1$ with the identification $v_{0,a} = v_{4m,a}$.
\begin{defi}
We say a pair of partitions $(P,Q)$ generic, if
\[
D-\left(\cup_{k=1}^{n_P} \Omega_P(k) \cup \cup_{k=1}^{n_Q} \Omega_Q(k)\right)
\]
does not contain a closed curve.
\end{defi}

\begin{lemm}\label{lem1}
For a generic pair of partitions $(P,Q)$ with any given colorings $c_P$ and $c_Q$, any connected component of $G_m(P,Q,c_P,c_Q)$  contains at least one of the following $2m$ vertices:
\[
v_{1,1},~ v_{3,1}, \ldots,~ v_{4m-3},~ v_{4m-1}.
\]
In particular, $G_m(P,Q,c_P,c_Q)$ has at most $2m$ connected components.
\end{lemm}
\begin{proof}
We first consider the case $m=1$. To claim $G_1(P,Q,c_P,c_Q)$ has only $2$ connected components, it is sufficient to prove that if $c_P(\Omega_P(a_1))=c_P(\Omega_P(a_2))$, then $v_{1,a_1}$ and $v_{1,a_2}$ are path-connected.

Because $(P,Q)$ is a generic pair, one can find a chain of open-sets
\[
\Omega_P(a_1)=\Omega_P(c_1),~\Omega_Q(b_1),~\Omega_P(c_2),~\Omega_Q(b_2),~\ldots, \Omega_P(c_k) = \Omega_P(a_2)
\]
such that two adjacent open-sets have non-trivial intersection.

Observe that if $\Omega_P(c) \cap \Omega_Q(b) \neq \emptyset$, then either
\[
\{v_{1,c}, v_{2,b}\}, ~ \text{or} ~\{v_{1,c}, v_{4,b}\}
\]
is an edge, and likewise either
\[
\{v_{3,c}, v_{2,b}\}, ~ \text{or} ~\{v_{3,c}, v_{4,b}\}
\]
is an edge.

Therefore the above chain of open-sets corresponds to a path connecting $v_{1,a_1}$ with either $v_{1,a_2}$ or $v_{3,a_2}$. However, from the assumption $c_P(\Omega_P(a_1))=c_P(\Omega_P(a_2))$, and from the construction of $G_1(P,Q,c_P,c_Q)$, $v_{1,a_1}$ cannot be connected to $v_{3,a_2}$, hence is connected to $v_{1,a_2}$.

Now for the rest, note that $G_m$ is an $m$-covering of $G_1$, and because $v_{1,1}$ and $v_{1,3}$ belongs to the different connected components of $G_1$, any connected components of $G_m$ must contain at least one vertex of the fiber of $v_{1,1}$ or $v_{1,3}$.
\end{proof}

For a large class of colorings, we can deduce a much stronger result.
\begin{lemm}\label{mainlem}
Let $(P,Q)$ be a generic pair of partitions. Assume that we are given with a pair of colorings $c_P$ and $c_Q$:

There exist four open sets $\Omega_P(a_1),\Omega_P(a_2), \Omega_Q(b_1),\Omega_Q(b_2)$ such that
\[
\Omega_P(a_i)\cap \Omega_Q(b_j) \neq \emptyset
\]
for $i=1,2$ and $j=1,2$, and that $c_P(\Omega_P(a_1))+c_P(\Omega_P(a_2))=c_Q(\Omega_Q(b_1))+c_Q(\Omega_Q(b_2))=1$.

Then the graph $G_m(P,Q,c_P,c_Q)$ has $2$ connected components.
\end{lemm}
\begin{proof}
Note that any connected component of $G_m$ must contain either one of $v_{4j,a_1}$ or one of $v_{4j,a_2}$ with $j=1,\ldots, m$, because $G_1$ has only two connected components.

Without loss of generality, assume that
\[
c_P(\Omega_P(a_1))=c_Q(\Omega_Q(b_1)).
\]
Then from the construction of the graph and from the assumption of the lemma
\begin{align*}
&\{v_{4j,a_1},v_{4j+1,b_1}\},\\
&\{v_{4j+1,b_1},v_{4j+2,a_2}\},\\
&\{v_{4j+2,a_2},v_{4j+3,b_2}\}, \\
\text{and} &\{v_{4j+3,b_2},v_{4j+4,a_1}\}
\end{align*}
are edges, hence $v_{4j, a_1}$  and $v_{4j+4, a_1}$ are connected. Likewise, $v_{4j, a_2}$ and $v_{4j+4, a_2}$ are connected. Therefore any connected component of $G_m$ must contain either $v_{4,a_1}$ or $v_{4,a_2}$.
\end{proof}
\subsection{The number of nodal domains of generic eigenfunctions}

Let $P$ be a principal $S^1$ bundle over a connected smooth compact Riemannian surface $X$ with the covering map $\pi: P \to X$. Let $m$ be a fixed integer, and assume that $\phi \in C^1(M)$ satisfies the following conditions:

\begin{condition}\label{cond}
For any small open $U\subset X$ such that $\pi^{-1}U \cong U \times S^1$, there exists a local coordinate $(x,\theta)$ of $\pi^{-1}U$ such that
\begin{enumerate}
\item[(i)]  $\phi(x,\theta) =  f(x) e^{i m \theta},$
\item[(ii)] the zero set of $\Re f$ (resp. $\Im f$) gives rise to a partition $P=P_U$ (resp. $Q=Q_U$) of $U$, and
\item[(iii)] $(P_U,Q_U)$ is a generic pair of partitions of $U$.
\end{enumerate}
\end{condition}

In this section, we prove the following theorem.
\begin{theo}\label{genthm}
Fix any point $x \in X$ such that $\phi (x,\theta) \neq 0$. Then any nodal domain of $\Re\phi $ has a nonempty intersection with $\pi^{-1}x$. In particular, the number of nodal domains of $\Re \phi$ is $\leq 2m$. Assume further that $\phi$ has a regular zero. Then the number of nodal domains of $\Re \phi$ is $2$.
\end{theo}

We begin with few observations in terms of fixed $U$ and a local coordinate $(x,\theta)$ of $\pi^{-1}U$.
\begin{prop}
If $\Re \phi$ is positive on two open sets $U_1 \subset \pi^{-1}U \cap \{\theta=\frac{k\pi}{2m}\}$ and $U_2 \subset \pi^{-1}U \cap \{\theta=\frac{(k+1)\pi}{2m}\}$ for some integer $k$, and if $\pi U_1 \cap \pi U_2 \neq \emptyset$, then $U_1$ and $U_2$ are contained in the same nodal domain of $\Re \phi$.
\end{prop}
\begin{proof}
Let $x_0$ be a point in the intersection $\pi U_1 \cap \pi U_2$. Then from the equation
\[
\Re \phi (x_0, \theta) = \Re f(x_0) \cos (m\theta) + \Im  f(x_0) \sin(m\theta),
\]
we see that $\Re \phi$ is positive along the curve
\[
\{(x_0,\theta):\frac{k \pi }{2m}\leq \theta \leq \frac{(k+1)\pi }{2m}\},
\]
which connects $U_1$ and $U_2$. Therefore $U_1$ and $U_2$ are contained in the same nodal domain.
\end{proof}
\begin{prop}
Any nodal domain of $\Re \phi|_{\pi^{-1}U}$ must intersect $\pi^{-1}U \cap \{\theta=\frac{k\pi}{2m}\}$ nontrivially for some integer $k \in \mathbb{Z}$.
\end{prop}
\begin{proof}
Assume for contradiction that $\Omega$ is a nodal domain of $\Re \phi|_{\pi^{-1}U}$ that is contained in
\[
\pi^{-1}U \cap \{\frac{k\pi}{2m}<\theta<\frac{(k+1)\pi}{2m}\}.
\]
From the equation
\[
\Re \phi (x, \theta) = \Re f(x) \cos (m\theta) + \Im  f(x) \sin(m\theta),
\]
we see that for each fixed $x$, $\Re \phi (x, \theta)$ either vanishes identically or has at most one sign change along the curve
\[
\{(x, \theta):\frac{k\pi}{2m}<\theta<\frac{(k+1)\pi}{2m}\}.
\]
This implies that if $x \in \pi \Omega$, then
\[
\Re \phi (x, \frac{k\pi}{2m}) = \Re \phi (x, \frac{(k+1)\pi}{2m})=0,
\]
which contradicts the assumption that the zero set of $\Re f$ gives rise to a partition of $U$.
\end{proof}
From these two propositions, we see that the nodal domains of $\Re \phi|_{\pi^{-1}U}$ can be understood from the nodal domains of the restrictions of $\Re \phi|_{\pi^{-1}U}$ to the $4m$-hypersurfaces
\[
\pi^{-1}U \cap \{\theta=\frac{k\pi}{2m}\}, ~ k =0,1,2,\ldots, 4m-1.
\]
In particular, if we define $c_{P_U}$ and $c_{Q_U}$ in terms of the sign of $\Re f$ and $\Im f$, then the number of connected components of $G_m(P_U,Q_U, c_{P_U}, c_{Q_U})$ is equal to the number of nodal domains of $\Re\phi|_{\pi^{-1}U}$.

\begin{proof}[Proof of Theorem \ref{genthm}]
Let $x\in X$ be a point where $\phi(x,\theta)\neq 0$, and let $U$ be a sufficiently small neighborhood of $x$. We may assume without loss of generality that the vertices
\[
v_{1,1},~ v_{3,1}, \ldots,~ v_{4m-3},~ v_{4m-1}
\]
of $G_m(P_U,Q_U, c_{P_U}, c_{Q_U})$ correspond to the nodal domains of the restrictions of $\Re \phi|_{\pi^{-1}U}$ to the hypersurfaces
\[
\pi^{-1}U \cap \{\theta=\frac{k\pi}{2m}\}, ~ k =1,3,\ldots,4m-3, 4m-1,
\]
that intersect the fiber $\pi^{-1}x$. Then Lemma \ref{lem1} implies that any nodal domain of $\Re \phi|_{\pi^{-1}U}$ must intersect $\pi^{-1}x$.

Now assume that $x'$ is another point in $U$. Then we may restate this as ``any nodal domain of $\Re \phi|_{\pi^{-1}U}$ that intersect $\pi^{-1}x'$ must intersect $\pi^{-1}x$'', and equivalently, ``any nodal domain of $\Re \phi$ that intersect $\pi^{-1}x'$ must intersect $\pi^{-1}x$''. Because we assumed that $X$ is connected, by the freedom of choice of the pair of points $x$ and $x'$, any nodal domain of $\Re \phi$ must intersect $\pi^{-1}x$. This proves the first part of the theorem.

For the latter part of the theorem, let $p$ be a regular zero of $\phi$, i.e.,
\[
d\phi : T_p P \to \mathbb{C}
\]
is a surjection. Choose a sufficiently small neighborhood $U \subset X$ of $\pi p$, and let $f$ be the function that satisfies
\[
\phi(x,\theta) =  f(x) e^{i m \theta} = \Re f \cos (m\theta) + \Im f \sin (m\theta) + i(\Im f \cos (m\theta) - \Re f \sin (m\theta)) = \Re \phi + i \Im \phi.
\]
If $d\Re f$ and $d \Im f$ are linearly dependent, then a straightforward computation implies that $d\phi$ has rank $\leq 1$, so $d\Re f$ and $d \Im f$ are linearly independent.

This implies that $\pi p$ is a regular zero of both $\Re f$ and $\Im f$. Also, linear independency implies that locally around $\pi p$, $\Re f=0$ and $\Im f=0$ define two curves intersecting transversally at $\pi p$. From this, we may find four open sets near $p$ that are required for Lemma \ref{mainlem}, and we infer that the number of nodal domains of $\Re\phi|_{\pi^{-1}U}$ is two.

Now because any nodal domain of $\Re\phi$ must intersect with $\pi^{-1}x$ for some $x \in U$, any nodal domain of $\Re\phi$ must contain one of the nodal domains of $\Re\phi|_{\pi^{-1}U}$, from which we conclude that $\Re \phi$ has only two nodal domains.
\end{proof}

We are ready to prove our main theorem, Theorem \ref{MAIN}.
\begin{proof}
It is sufficient to verify the assumptions in Theorem \ref{genthm} is satisfied. The first condition is trivial to verify. For the other conditions, note from the assumption that $P \to X$ is non-trivial, $\zcal_{f_{m,j}}$ is non-empty, and Theorem \ref{GENERIC1} implies that it is discrete and consists only of regular zeros.
\end{proof}

\begin{rema}
If $\zcal_{f_{m,j}}$ contains a closed curve that divides $X$ into two connected components, then the number of nodal domain can be large. For instance, if $f_{m,j}$ vanishes on the boundary of small open disc $U \subset X$, and if it does not vanish on $U$, then $\Re\phi_{m,j}$ vanish identically on $\partial \left( \pi^{-1}U \right)$, and therein, $\Re\phi_{m,j}$ has $2m$-distinct nodal domains. In particular, Theorem \ref{genthm} fails even if $f_{m,j}$ has a regular zero elsewhere.
\end{rema}

\section{\label{SCCSECT} Surfaces of constant curvature}
In this section, we  illustrate the geometry of Kaluza-Klein metrics and
the Kaluza-Klein eigenvalue problem on unit tangent bundles of surfaces of constant curvature.

\subsection{\label{FLATTORISECT} Flat tori} Let $\T^2 = \R^2/\Z^2$. We use coordinates
$z = x_1 + i x_2$.  Its unit tangent bundle
is $S \T^2 = \T^2 \times S^1$. The connection is flat and
$\Delta_H = \Delta$ is simply the Laplacian of $\T^2$. The Kaluza-Klein Laplacian is that $\Delta_G = \Delta + \frac{\partial^2}{\partial \theta^2}$ on $\T^2 \times S^1$. The Kaluza-Klein eigenfunctions are linear combinations of the product
eigenfunctions,
$$\phi_{m, \vec{k}}(x_1, x_2, \theta) = e^{i \langle \vec k, \vec x \rangle} e^{i m \theta}, \;\; \Delta_G \phi_{m, \vec{k}} = - (|\vec k|^2 + m^2)  \phi_{m, \vec{k}}. $$ The multiplicity of the eigenvalue with fixed $m$ is the number of ways of representing $|\vec k|^2$ as a sum of two squares. They correspond to eigendifferentials
$$f_{m, \vec{k}}(z) (dz)^m =  e^{i \langle \vec k, \vec x \rangle} (dz)^m.$$

In the notation \eqref{uv},
$$\left\{ \begin{array}{l} \Re \phi_{m, \vec{k}}(x_1, x_2, \theta)  = u_{m, \vec{k}}(x_1, x_2, \theta) = \cos (\langle \vec k, \vec x \rangle + m\theta), \\ \\ \;\; \Im  \phi_{m, \vec{k}}(x_1, x_2, \theta)  =v_{m, \vec{k}}(x_1, x_2, \theta)  = \sin (\langle \vec k, \vec x \rangle +m \theta). \end{array} \right.$$

The nodal sets of the imaginary part are given by,
$$\zcal_{v_{m, \vec{k}} } = \{(x_1, x_2, \theta):
\langle \vec k, \vec x \rangle +m \theta \in \pi \Z \}. $$
$\zcal_{v_{m, \vec{k}} }$ contains the set  $$
 \{(x_1, x_2, \theta): \langle \vec k, \vec x \rangle \in \pi \Z, \theta = \ell \frac{ \pi}{m}, \ell =1, \dots, m\}. $$

 Note that  $\phi_{m, \vec{k}}(x_1, x_2, \theta) $ has no zeros on $\T^2 \times S^1$ and $f_{m, \vec{k}}(z) (dz)^m$ has no zeros as an $m$-differential
 on $\T^2$.

 If we change the lattice to a general lattice $L \subset \R^2$, the
 eigenfunctions of $\T^2$ change to $e_{\vec \lambda}(\vec x) = e^{2 \pi i \langle \vec \lambda, \vec x}$ where $\vec \lambda \in \Lambda = L^*$, the dual lattice. For generic $L$, the eigenvalues have multiplicity $2$ and
 the eigenspaces are spanned by the real and imaginary parts of
 $e_{\vec \lambda}$ or equivalently by $e_{\vec \lambda}$ and its complex conjugate
 $e_{-\vec \lambda}$. The same is true of the Kaluza-Klein eigenfunctions
 $\phi_{m, \vec \lambda} = e^{2 \pi i \langle \vec \lambda, \vec x  \rangle} e^{i m \theta}$. Again, $\phi_{m, \vec \lambda}$ has no zeros. Using the bifurcation of nodal sets of eigenfunctions under generic paths of metrics
 of \cite{U}, one can show that

  \begin{conjecture} for generic Kaluza-Klein metrics on $S \T^2$,
 the joint eigenfunctions $\phi_{m,j}$ have no zeros.
 \end{conjecture}

We now give an explicit orthonormal eigenbasis of $\mathbb{T}^3$ such that all of them have exactly two nodal domains, hence proving Theorem \ref{explicit}.

To begin with, let $f_1 (x) = \cos (2\pi x)$ and $f_0 (x) = \sin (2 \pi x)$. Then
\[
\{f_{j_1}(m_1 x_1)f_{j_2}(m_2 x_2)f_{j_3}(m_3 x_3)~:~ j_k =0 \text{ or } 1,~ m_k \in \mathbb{Z}_{\geq 0}\}
\]
is an orthogonal eigenbasis of $\mathbb{T}^3$. We consider four cases.

\paragraph{Case 1: $m_1 m_2 m_3 >0$} We first have
\begin{multline*}
\langle \{f_{j_1}(m_1 x_1)f_{j_2}(m_2 x_2)f_{j_3}(m_3 x_3), f_{1-j_1}(m_1 x_1)f_{1-j_2}(m_2 x_2)f_{1-j_3}(m_3 x_3)\} \rangle \\
=\langle \{f_{j_1}(m_1 x_1)f_{j_2}(m_2 x_2)f_{j_3}(m_3 x_3)\pm f_{1-j_1}(m_1 x_1)f_{1-j_2}(m_2 x_2)f_{1-j_3}(m_3 x_3)\} \rangle
\end{multline*}
Assume without loss of generality that $j_1 =0$. Then
\begin{multline*}
f_{j_1}(m_1 x_1)f_{j_2}(m_2 x_2)f_{j_3}(m_3 x_3)\pm f_{1-j_1}(m_1 x_1)f_{1-j_2}(m_2 x_2)f_{1-j_3}(m_3 x_3)\\
=\Re\left(\left(f_{j_2}(m_2 x_2)f_{j_3}(m_3 x_3)\pm i f_{1-j_2}(m_2 x_2)f_{1-j_3}(m_3 x_3)\right) e^{2\pi i m_1 x_1}\right),
\end{multline*}
has two nodal domains by Theorem \ref{genthm}, because
\[
f_{j_2}(m_2 x_2)f_{j_3}(m_3 x_3)\pm i f_{1-j_2}(m_2 x_2)f_{1-j_3}(m_3 x_3)
\]
has a regular zero.

\paragraph{Case 2: exactly one $m_k$ is zero, and the other two are different} From the same reasoning, each eigenfunction in the new basis in the following has two nodal domains:
\begin{multline*}
\langle \{f_{j_1} (m_1 x_1) f_{j_2} (m_2 x_2),~ f_{j_1} (m_1 x_1) f_{j_3} (m_2 x_3)~:~ j_k=0 \text{ or } 1\}\rangle\\
=\langle \{f_{j_1} (m_1 x_1) f_{j_2} (m_2 x_2) \pm f_{1-j_1} (m_1 x_1) f_{j_3} (m_2 x_3)~:~ j_k=0 \text{ or } 1\}\rangle,
\end{multline*}
\begin{multline*}
\langle \{f_{j_2} (m_1 x_2) f_{j_1} (m_2 x_1),~ f_{j_2} (m_1 x_2) f_{j_3} (m_2 x_3)~:~ j_k=0 \text{ or } 1\}\rangle\\
=\langle \{f_{j_2} (m_1 x_2) f_{j_1} (m_2 x_1) \pm f_{1-j_2} (m_1 x_2) f_{j_3} (m_2 x_3)~:~ j_k=0 \text{ or } 1\}\rangle,
\end{multline*}
and
\begin{multline*}
\langle \{f_{j_3} (m_1 x_3) f_{j_1} (m_2 x_1),~ f_{j_3} (m_1 x_3) f_{j_2} (m_2 x_2)~:~ j_k=0 \text{ or } 1\}\rangle\\
=\langle \{f_{j_3} (m_1 x_3) f_{j_1} (m_2 x_1) \pm f_{1-j_3} (m_1 x_3) f_{j_2} (m_2 x_2)~:~ j_k=0 \text{ or } 1\}\rangle.
\end{multline*}

\paragraph{Case 3: exactly one $m_k$ is zero, and the other two are equal} Again by the same reasoning, each of the following
\begin{align*}
f_0 (m x_1) f_0(mx_2) \pm f_1(mx_1) f_0(mx_3), \\
f_0 (m x_2) f_0(mx_3) \pm f_1(mx_2) f_0(mx_1),\\
f_0(mx_3) f_0(mx_1) \pm f_1(mx_3) f_0 (mx_2),\\
f_1 (m x_1) f_1(mx_2) \pm f_1(mx_3)f_0(mx_1),  \\
f_1 (m x_2) f_1(mx_3) \pm f_1(mx_1)f_0(mx_2), \\
f_1(mx_3) f_1(mx_1) \pm f_1 (mx_2)f_0(mx_3)
\end{align*}
has two nodal domains, and these are the basis of
\[
\langle \{f_{j_1}(mx_1)f_{j_2}(mx_2),f_{j_1}(mx_1)f_{j_3}(mx_3),f_{j_2}(mx_2)f_{j_3}(mx_3)~:~ j_k=0 \text{ or } 1\}\rangle.
\]

\paragraph{Case 4: exactly one $m_k$ is nonzero} In this case, we consider orthogonal eigenfunctions
\begin{align*}
f_0 (m x_1) + f_0 (m x_2)-\frac{1}{2}f_0 (m x_3),\\
f_0 (m x_1) + f_0 (m x_3)-\frac{1}{2}f_0 (m x_2),\\
f_0 (m x_2) + f_0 (m x_3)-\frac{1}{2}f_0 (m x_1),\\
f_1 (m x_1) + f_1 (m x_2)-\frac{1}{2}f_1 (m x_3),\\
f_1 (m x_1) + f_1 (m x_3)-\frac{1}{2}f_1 (m x_2),\\
f_1 (m x_2) + f_1 (m x_3)-\frac{1}{2}f_1 (m x_1),
\end{align*}
which span
\[
\langle \{f_{j}(mx_1),f_{j}(mx_2),f_{j}(mx_3)~:~ j=0 \text{ or } 1\}\rangle.
\]
Each of these has only two nodal domains from the following lemma.
\begin{lemm}
Let $m$ be a positive integer. Then
\[
\cos (m x_1) + \cos (m x_2) - \frac{1}{2} \cos (m x_3)
\]
has only two nodal domains.
\end{lemm}
\begin{proof}
Let $x_1-x_2=a$, $x_1-x_3=b$, and $x_2+x_3 = c$. Then
\begin{multline*}
e^{2\pi i mx_1} + e^{2\pi i mx_2} -\frac{1}{2} e^{2\pi i mx_3}\\
 = \left(e^{\pi i m a}e^{\pi i m b } + e^{-\pi i m a}e^{\pi i m b } -\frac{1}{2} e^{\pi i m a}e^{-\pi i m b }\right) e^{2\pi i m c},
\end{multline*}
and from Theorem \ref{genthm}, it is sufficient to prove that
\[
e^{\pi i m a}e^{\pi i m b } + e^{-\pi i m a}e^{\pi i m b } -\frac{1}{2} e^{\pi i m a}e^{-\pi i m b }
\]
has a regular zero. Let $\pi m (a+b)=x$ and $\pi m (a-b)=y$, then this is equivalent to
\[
\cos x + \frac{1}{2} \cos y + i \left(\sin x - \frac{3}{2} \sin y\right)
\]
having a regular zero. Since $\cos x + \frac{1}{2} \cos y$ and $\sin x - \frac{3}{2} \sin y$ do not have singular points, it is sufficient to check if these two functions have a common zero, in other words, if
\[
\cos x + \frac{1}{2} \cos y + i \left(\sin x - \frac{3}{2} \sin y\right)= 0
\]
has a solution. Note that this is equivalent to
\begin{equation}\label{something}
e^{ix} = -\frac{1}{2} \cos y + i \frac{3}{2} \sin y.
\end{equation}
Because
\[
\left|-\frac{1}{2} \cos y + i \frac{3}{2} \sin y\right| = \frac{1}{4} + 2 \sin^2 y,
\]
for $y$ such that   $\frac{1}{4} + 2 \sin^2 y=1$, there is $x$ satisfying \eqref{something}, and this completes the proof.
\end{proof}

\subsection{\label{SPHERE} Kaluza-Klein metrics on $S^3$ }

Let $(S^2, g_0)$ be the $2$-sphere with its standard metric of curvature $1$.
Then its unit tangent  $S S^2 = SO(3) =  \RP^3 = S^3/\pm 1$ and the Kaluza-Klein metric
is the standard metric of constant sectional curvature $1$ on $S^3$ (divided
by the antipodal group $\Z_2$). The Kaluza-Klein Laplacian is therefore the
standard Laplacian $\Delta_{S^3}$ on $\Z_2$-invariant functions.

Since $S^3$ is a group, $L^2(S^3) = \bigoplus_{N=0}^{\infty} V_N \otimes V_N$ where $V_N $ is an irreducible representation of $S^3$ of dimension $N+1$.  Alternatively, the eigenfunctions of
$S^3$ are harmonic homogeneous polynomials on $\R^3$.
Moreover, $\Delta |_{V_N \otimes V_N} =  N(N+2) = (N+1)^2 - 1$. The eigenfunctions of $\RP^3$ are those where $N$ is even.

We need explicit separation of variables expressions for equivariant
spherical harmonics, and therefore need to introduce coordinate systems.
We use `axis - angle'  Hopf coordinates $(\alpha, \theta, \phi)$  on $S^3 \subset \R^4$ defined by
$$\vec x = \begin{pmatrix}x_1 \\ x_2 \\ x_3 \\ x_4 \end{pmatrix} = r \begin{pmatrix} \sin \alpha  \cos \phi \\
\sin \alpha \sin \phi  \\
\cos \alpha  \cos \theta  \\
\cos \alpha \sin \theta \end{pmatrix}.$$
Here $0 \leq \alpha \leq  \pi/2, 0 \leq \theta, \phi  \leq 2 \pi$.  This corresponds
to writing $$z_1 = e^{i \phi} \sin \alpha, \;\; z_2 = e^{i \theta} \cos \alpha, ((z_1, z_2) \in \C^2 \simeq \R^4)). $$
There exist two commuting isometric $S^1$ actions generated by the Killing vector fields
$$X = \frac{\partial}{\partial \phi} + \frac{\partial}{\partial \theta},\;\;
Y = \frac{\partial}{\partial \phi}  - \frac{\partial}{\partial \theta}. $$

The metric is
$(d \alpha)^2 + (\cos \alpha d \theta)^2 + (\sin \alpha d \phi)^2. $
In these coordinates one has an orthogonal basis of eigenfunctions given by
\begin{multline*}
\Phi^{m_+, m_-}_N (\alpha, \phi, \theta)  =   C^{m_+, m_-}_N\;e^{i (m_+ + m_-) \phi} e^{i(m_+-m_-) \theta} \cdot\\
\cdot  (1- \cos 2 \alpha)^{\frac{m_+ + m_-}{2}} (1+ \cos 2 \alpha)^{\frac{m_+ - m_-}{2}}  P_{\frac{N}{2} - m_+}^{m_+ + m_-, m_+ - m_-}(\cos 2 \alpha),
\end{multline*}
where $P_N^{(a,b)}$ is a Jacobi polynomial and where
$$|m_{\pm}| \leq \frac{N}{2},  \frac{N}{2} - m_{\pm} \in {\mathbb N}. $$Here, weight  $m$ in our sense  means that the eigenfunctions transform by $e^{i m_-(\phi - \theta)}$.
 $\Phi^{m_+, m_-}_N$  are also known  as ``Wigner D-functions'' on $SU(2)$. Another expression is
\begin{multline*}
T_N^{m_1, m_2} = \\
C_N^{m_1, m_2}\; (\cos \alpha e^{i \theta})^{m_1 + m_2} (\sin \alpha e^{i \phi})^{m_2-m_1} P_{N/2 - m_2}^{(m_2 - m_1, m_2 + m_1)} (\cos (2 \alpha)).
\end{multline*}
Here $m_1 = m$ in our notation.
These are manifestly joint eigenfunctions of $\Delta_{S^3}$ and
 of $\frac{\partial}{\partial \theta}, \frac{\partial}{\partial \phi}. $

 \begin{lemm} \label{S3BAD} The nodal sets of the  equivariant eigenfunctions  $\Phi^{m_+, m_-}_N$ (or
 equivalently $T_N^{m_1, m_2} $)  have real dimension $2$.
 \end{lemm}

 \begin{proof}
The   only factors with
 zeros are the $\alpha$-functions. These have roughly $m$ discrete zeros
 in $\alpha$. Hence, the complex nodal set is a union
 $$\{(\theta, \phi, \alpha): (\cos \alpha)^{m_1 + m_2} (\sin \alpha)^{m_2-m_1} P_{N/2 - m_2}^{(m_2 - m_1, m_2 + m_1)} (\cos (2 \alpha)) = 0\}, $$
 and thus has real dimension $2$.\end{proof}

   As a result, these eigenfunctions do not satisfy the conditions of the generic Kaluza-Klein metrics to which our results apply,
 and their nodal sets are quite different.

 As mentioned in the introduction,   the numerical experiments of A. Barnett et all \cite{B18}
 show that random spherical harmonics of degree $N$ on $S^3$ also
 have different types of nodal sets than our generic eigenfunctions. Namely,
 the expected  number of nodal domains has the asymptotics $c N^3$ for a certain $c > 0$. As proved in \cite{JZ2019}, the  nodal sets of real/imaginary
 parts of random equivariant eigenfunctions with fixed $m$ (a subspace isomorphic to $V_N$) have connected nodal sets.
 The difference is due to the fact that our random equivariant spherical harmonics are a thin subset of the
 random spherical harmonics of degree $N$ on $S^3$.
\begin{rema}
In \cite{JZ2019}, we compute the expected genus of the single component of
the nodal sets of real/imaginary parts of  random equivariant spherical harmonics of degree $(N,m)$ where $|m| \leq N, 2 | (N-m)$.  The expected Euler characteristic is of the form $m (N^2 - m^2) + N$ modulo lower order terms.
\end{rema}

\subsection{\label{HYPERBOLICSECT} Hyperbolic surfaces ${\mathbb H}^2$}

Although it differs from our prior discussion in the compact case, let
us consider  a  finite area hyperbolic real Riemann surface of constant negative curvature $-1$.  Then $X = S^*_g M = \Gamma \backslash G$ where $G  = PSL(2, \R)$. The total space $X$ carries a Lorentz Cartan-Killing metric with indefinite Laplacian the Casimir operator $\Omega$.
It is well known that $\Omega = H^2 + V^2 - W^2$. We now change the
sign of the third term to get the Kaluza-Klein Laplacian $\Delta_{X} = H^2 + V^2 + W^2$. The associated metric defines a Riemannian submersion
$\pi: X \to M$ with fibers given by $K$-orbits. They are necessarily totally geodesic. It follows that the horizontal Laplacian $H^2 + V^2$ commutes
with the vertical Laplacian $W^2$. This is obvious because $0 = [\Omega, W^2] = [H^2 + V^2, W^2] = 0$.

The joint eigenfunctions of $\Omega, W$ are denoted by $\phi_{m,j}$.
When $ m = 0$ they are pullbacks of eigenfunctions of $M = \Gamma \backslash G / K$.

In particular the number of nodal domains of $\phi_{j, 0}$ on $X$ is the same as the number of nodal domains of $\phi_{j}$ on $M$. The former nodal sets are $K$-invariant and in the case of regular nodal components
are $2$-tori over circles.

 The lift of weight $m$ of an $m$-differential $f (dz)^m$ is given by
$$\Phi(x, y, \theta) = y^{m/2} f(x + iy) e^{-i m \theta}. $$
Here, the \kahler potential is
 $\phi =  \log y$, $d  \phi = \frac{dy}{y}$, $\Delta \phi = y^2 (\log y)'' = - 1$. Also,
 $\|d \phi\|^2 = y^2 \|\frac{dy}{y}\|^2 = 1$ and
$* d \phi = * (\phi_x dx  + \phi_y dy) = (- \phi_x dy + \phi_y dx) y$. The Maass operator is $$D_m =  y^2 (\frac{\partial^2}{\partial x^2} + \frac{\partial^2}{\partial y^2}) - 2 i m y\frac{\partial }{\partial x} $$
and
$$D_m f_{m,j} = s(1 -s) f_{m,j}. $$


Breaking up into real and imaginary parts gives the system,
$$\left\{ \begin{array}{l} \Delta \Re f_{m,j} + 2 m y\frac{\partial}{\partial x} \Im  f_{m,j} = s(1-s) \Re f_{m,j}, \\ \\
\Delta \Im f_{m,j} - 2 m y\frac{\partial}{\partial x}  \Re f_{m,j} = s(1-s) \Im f_{m,j}.  \end{array} \right.$$


The raising/lowering operators are  the Maass operators defined by
\begin{equation*} \left\{ \begin{array}{l} K_k  = (z - \bar{z}) \frac{\partial}{\partial z} + k
= 2 i y^{1-k} \frac{\partial}{\partial z} y^k,  \\ \\
 L_k  = (z - \bar{z}) \frac{\partial}{\partial z} - k
= -2 i y^{1+k} \frac{\partial}{\partial \bar{z}} y^{-k} = \overline{K}_{-k}. \end{array} \right. \end{equation*}
Then,
$$K_k: \mathfrak H_{k} \to \mathfrak H_{k+1},\;\; L_k: \mathfrak H_{k} \to \mathfrak H_{k-1},$$
and $$D_{k + 1} K_k = K_k D_k, D_k L_{k + 1} = L_{k + 1} D_k, $$
and
$$D_k = L_{k+1} K_k + k(k +1) = K_{k-1} L_k + k(k-1). $$

\subsubsection{Automorphic forms on the full modular group}
Now we consider the case $\Gamma = PSL_2(\mathbb{Z})$. Note that the quotient $\Gamma \backslash G$ is non-compact in this case. Nevertheless, it is known that $-\Delta_G$ has infinitely many discrete spectrum, where corresponding $L^2$ integrable eigenfunctions can be chosen so that they are in one-to-one correspondence with Maass--Hecke cusp forms or holomorphic Hecke cusp forms. We refer the readers to \cite{MR2061214} for detailed background.
\begin{theo}
Let $X = PSL_2(\mathbb{Z})\backslash \mathbb{H}$, and let $\phi_{m,ir}$ be a weight $m$ Maass--Hecke cusp form on $$PSL_2(\mathbb{Z})\backslash PSL_2 (\mathbb{R}).$$ Assume that the zeros of $\phi_{m,ir}$ are isolated. Then $\Re \phi_{m,ir}$ has only two nodal domains.
\end{theo}
\begin{proof}
The first statement of Condition \ref{cond} follows from the definition of Maass--Hecke cusp form, and the second statement follow from the fact that $\phi_{m,ir}:X \to \mathbb{C}$ can not be scaled to a real-valued function, and that $\phi_{m,ir}$ is analytic. The third statement follows from the assumption.

Now, because the first Hecke eigenvalue is $1$, the first Fourier coefficient of $\phi_{m,ir}$ at the cusp does not vanish, meaning that $i\infty$ is a regular zero of $\phi_{m,ir}$. We conclude the proof by applying Theorem \ref{genthm}.
\end{proof}
\begin{rema}
It is not hard to see that in the constant curvature case,  the nodal set of $\phi_{2,ir}$ consists of the fibers over the critical point set  $C_{\phi_{ir}}$  of $\phi_{ir}$.
At this time, it does not seem to be known whether  $C_{\phi_{ir}}$ is necessarily a discrete set of points in the case
of hyperbolic surfaces. This cannot be proved by a purely local calculation, since the critical point set of rotationally
invariant Dirichlet/Neumann eigenfunctions on a  compact rotationally invariant submanifold $C_R$ of a  hyperbolic cylinder ${\mathbb H}^2/\langle \gamma_0 \rangle$ consists of a union of $S^1$ orbits. Here, $\gamma_0$ is a hyperbolic element
and $\langle \gamma_0 \rangle$ is the cyclic group it generates. Thus, negative curvature does not rule out codimension $1$
critical point sets. One can put any negatively curved $S^1$ invariant metric on  $C_R$ and obtain the same result, so it
is not an effect of constant curvature. We conjecture that for compact hyperbolic surfaces without boundary,
 $C_{\phi_{ir}}$ is a finite set for every eigenfunction.
\end{rema}

When we have holomorphicity of $\phi$, we may remove the assumption that the zeros of $\phi$ being isolated. For instance, we have:
\begin{theo}\label{hecke}
Let $X = PSL_2(\mathbb{Z})\backslash \mathbb{H}$, and let $\phi_{m,0}$ be a Laplacian eigenfunction on $PSL_2(\mathbb{Z})\backslash PSL_2 (\mathbb{R})$ corresponding to a holomorphic Hecke cusp form $F$ of weight $m$. Then $\Re \phi_{m,0}$ has only two nodal domains.
\end{theo}
\begin{proof}
We first note that $\phi_{m,0} (z,\theta)= y^{m/2}F(z)e^{-im\theta}$, and $F$ is holomorphic. Therefore Condition \ref{cond} is satisfied.

Because we assumed that $F$ is a Hecke cusp form, the first Hecke eigenvalue is $1$. Therefore $i\infty$ is a regular zero of $\phi_{m,0}$, and now the theorem follows from Theorem \ref{genthm}.
\end{proof}
\begin{coro}
There exist eigenfunctions on $PSL_2 (\mathbb{Z}) \backslash PSL_2(\mathbb{R})$ that have only two nodal domains but with arbitrarily large eigenvalues.
\end{coro}

We remark here that Theorem \ref{hecke} is false, without the assumption that $F$ is a Hecke cusp form. To construct a counter example, let $\Delta(z)$ be the discriminant modular form given by
\begin{multline*}
\Delta(z) = \sum_{n=1}^\infty \tau(n) q^n \\
= q -24q^2 +252 q^3-1472 q^4  +4830 q^5  -6048 q^6  -16744 q^7  +\ldots,
\end{multline*}
where $q=e^{2\pi i z}$. This is a weight $12$ modular form on $PSL_2(\mathbb{Z})\backslash \mathbb{H}$. Thus $\Delta(z)^2$ is a modular form of weight $24$ on $PSL_2(\mathbb{Z})\backslash \mathbb{H}$, and
\[
\Phi=\Re (y^{12} \Delta(z)^2 e^{-24i\theta})
\]
is a Laplacian eigenfunction on $PSL_2(\mathbb{Z})\backslash PSL_2 (\mathbb{R})$ of weight $24$. To count the number of nodal domains of this eigenfunction, we let
\[
\mathcal{F}=\{x+iy~:~ |x| \leq \frac{1}{2}, ~ x^2+y^2 \geq 1\} \subset \mathbb{H}
\]
be the fundamental domain of $PSL_2(\mathbb{Z}) \backslash \mathbb{H}$, and let $M_0 = \mathcal{F} \times \{\theta~:~ 0\leq \theta < 2\pi\}$.

We then consider the restrictions of $\Phi$ to the top $\theta = 2\pi$, side $x=-1/2$, and front $x^2+y^2=1$ of the solid $M_0$.

It can be shown that the nodal set of $\Phi$ on the side is that of $\cos (24\theta)=0$, and on the front is that of $\cos (12(\varphi+2\theta))=0$, where we define $\varphi = \arccos (x)$. We compute the nodal set of the restriction to the top numerically using Mathematica.

\begin{center}
\includegraphics[width=\textwidth]{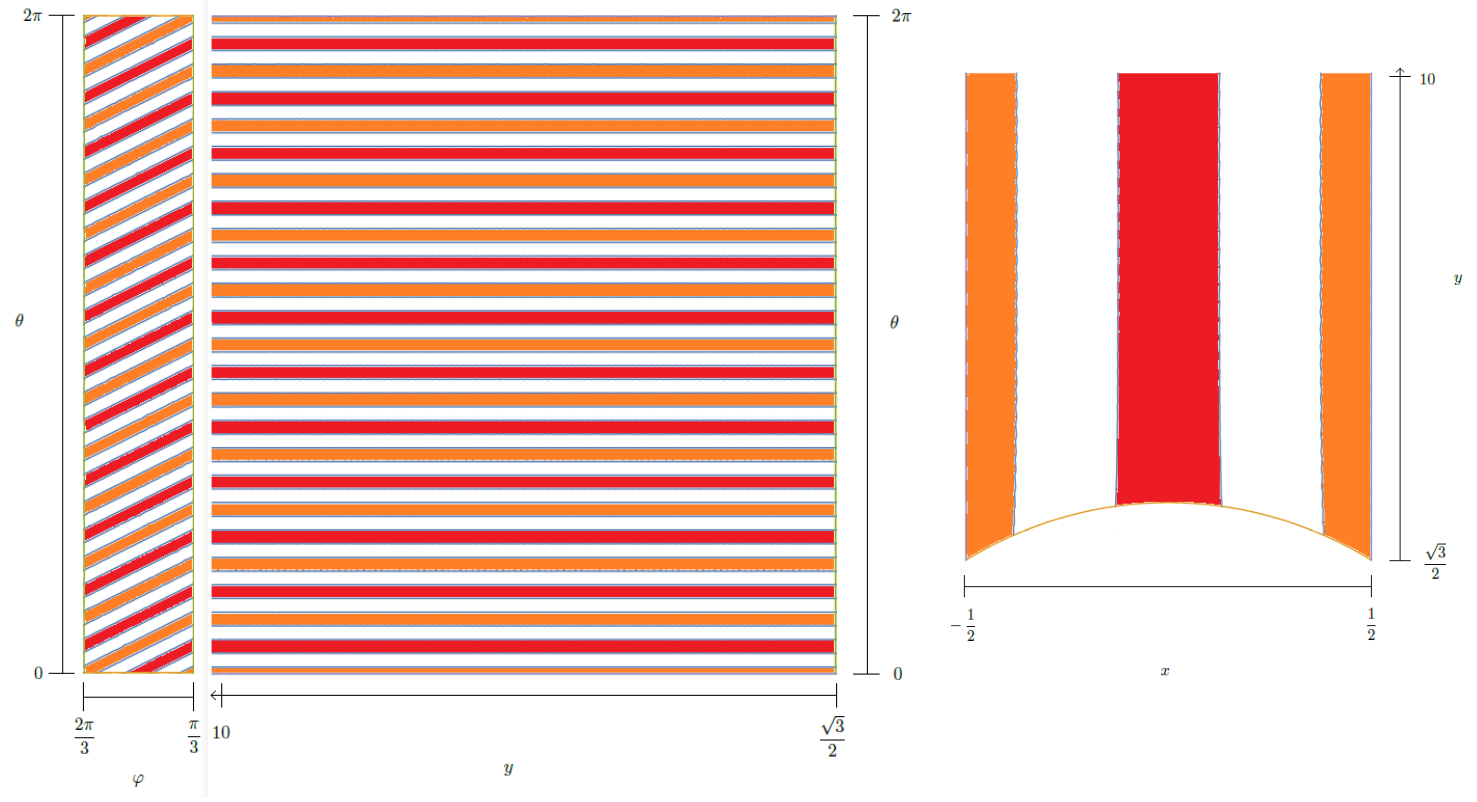}\\
The nodal set of $\Phi$ on the front, the side, and the top of the solid $M_0$.
\end{center}

Note that we may obtain $PSL_2(\mathbb{Z})\backslash PSL_2 (\mathbb{R})$ from $M_0$ by gluing the sides via $(x,y,\theta) = (x+1,y,\theta)$ (corresponding to $\begin{pmatrix}  1 & 1 \\ 0 & 1 \end{pmatrix}$), the top and the bottom via $(x,y,\theta) = (x,y,\theta+2\pi )$ (corresponding to $k(\theta)=k(\theta+2\pi)$), and then the front with itself via $(\varphi,\theta) = (\pi-\varphi, \theta+\varphi)$ and $(\varphi, \theta) = (\varphi, \theta+2\pi)$ (corresponding to $\begin{pmatrix}  0 & -1 \\ 1 & 0 \end{pmatrix}$ and $k(\theta)=k(\theta+2\pi)$).

From these, one can verify that $\Phi$ has exactly four nodal domains, where in the pictures above, two positive nodal domains are colored differently with red and orange.

\bibliographystyle{cdraifplain}
\bibliography{bibfile}

\def\bysame{\leavevmode ---------\thinspace}
\makeatletter\if@francais\providecommand{\og}{<<~}\providecommand{\fg}{~>>}
\else\gdef\og{``}\gdef\fg{''}\fi\makeatother
\def\cdrandname{\&}
\providecommand\cdrnumero{no.~}
\providecommand{\cdredsname}{eds.}
\providecommand{\cdredname}{ed.}
\providecommand{\cdrchapname}{chap.}
\providecommand{\cdrmastersthesisname}{Memoir}
\providecommand{\cdrphdthesisname}{PhD Thesis}
\begin{thebibliography}{10}

\bibitem{AM17}
{\scshape F.~A. Arias {\normalfont \cdrandname}~M.~Malakhaltsev}, {\og A
  generalization of the {G}auss-{B}onnet-{H}opf-{P}oincar\'e formula for
  sections and branched sections of bundles\fg}, \emph{J. Geom. Phys.}
  \textbf{121} (2017), p.~108-122.

\bibitem{B18}
{\scshape A.~Barnett, K.~Konrad {\normalfont \cdrandname}~M.~Jin}, {\og
  Experimental {N}azarov-{S}odin constants, genus, and percolation on nodal
  domains for $2${D} and $3${D} random waves\fg}, \emph{in preparation} (2017).

\bibitem{BBB}
{\scshape L.~B\'erard-Bergery {\normalfont \cdrandname}~J.-P. Bourguignon},
  {\og Laplacians and {R}iemannian submersions with totally geodesic
  fibres\fg}, \emph{Illinois J. Math.} \textbf{26} (1982), \cdrnumero 2,
  p.~181-200.

\bibitem{CH}
{\scshape R.~Courant {\normalfont \cdrandname}~D.~Hilbert}, \emph{Methods of
  mathematical physics. {V}ol. {I}}, Interscience Publishers, Inc., New York,
  N.Y., 1953, xv+561~pages.

\bibitem{D}
{\scshape J.-P. Demailly}, \emph{Analytic methods in algebraic geometry},
  Surveys of Modern Mathematics, vol.~1, International Press, Somerville, MA;
  Higher Education Press, Beijing, 2012, viii+231~pages.

\bibitem{EP12}
{\scshape A.~Enciso {\normalfont \cdrandname}~D.~Peralta-Salas}, {\og
  Nondegeneracy of the eigenvalues of the {H}odge {L}aplacian for generic
  metrics on 3-manifolds\fg}, \emph{Trans. Amer. Math. Soc.} \textbf{364}
  (2012), \cdrnumero 8, p.~4207-4224.

\bibitem{FN12}
{\scshape T.~Fukui {\normalfont \cdrandname}~J.~J. Nu\~no Ballesteros}, {\og
  Isolated singularities of binary differential equations of degree {$n$}\fg},
  \emph{Publ. Mat.} \textbf{56} (2012), \cdrnumero 1, p.~65-89.

\bibitem{GRS13}
{\scshape A.~Ghosh, A.~Reznikov {\normalfont \cdrandname}~P.~Sarnak}, {\og
  Nodal domains of {M}aass forms {I}\fg}, \emph{Geom. Funct. Anal.} \textbf{23}
  (2013), \cdrnumero 5, p.~1515-1568.

\bibitem{GRS15}
\bysame , {\og Nodal domains of {M}aass forms, {II}\fg}, \emph{Amer. J. Math.}
  \textbf{139} (2017), \cdrnumero 5, p.~1395-1447.

\bibitem{MR2061214}
{\scshape H.~Iwaniec {\normalfont \cdrandname}~E.~Kowalski}, \emph{Analytic
  number theory}, American Mathematical Society Colloquium Publications,
  vol.~53, American Mathematical Society, Providence, RI, 2004, xii+615~pages.

\bibitem{JS}
{\scshape S.~u. Jang {\normalfont \cdrandname}~J.~Jung}, {\og Quantum unique
  ergodicity and the number of nodal domains of eigenfunctions\fg}, \emph{J.
  Amer. Math. Soc.} \textbf{31} (2018), \cdrnumero 2, p.~303-318.

\bibitem{JY}
{\scshape J.~Jung {\normalfont \cdrandname}~M.~P. Young}, {\og Sign changes of
  the {E}isenstein series on the critical line\fg}, \emph{Int. Math. Res. Not.
  IMRN} (2019), \cdrnumero 3, p.~641-672.

\bibitem{JZ13}
{\scshape J.~Jung {\normalfont \cdrandname}~S.~Zelditch}, {\og Number of nodal
  domains and singular points of eigenfunctions of negatively curved surfaces
  with an isometric involution\fg}, \emph{J. Differential Geom.} \textbf{102}
  (2016), \cdrnumero 1, p.~37-66.

\bibitem{JZ14}
\bysame , {\og Number of nodal domains of eigenfunctions on non-positively
  curved surfaces with concave boundary\fg}, \emph{Math. Ann.} \textbf{364}
  (2016), \cdrnumero 3-4, p.~813-840.

\bibitem{JZ2019}
\bysame , {\og Topology of the nodal set of random equivariant spherical
  harmonics on $\mathbb{S}^3$\fg}, \emph{in preparation} (2019).

\bibitem{Lew77}
{\scshape H.~Lewy}, {\og On the minimum number of domains in which the nodal
  lines of spherical harmonics divide the sphere\fg}, \emph{Comm. Partial
  Differential Equations} \textbf{2} (1977), \cdrnumero 12, p.~1233-1244.

\bibitem{MM}
{\scshape M.~Magee}, {\og Arithmetic, zeros, and nodal domains on the
  sphere\fg}, \emph{Comm. Math. Phys.} \textbf{338} (2015), \cdrnumero 3,
  p.~919-951.

\bibitem{CH2}
{\scshape A.~Stern}, {\og {Bemerkungen \"uber asymptotisches Verhalten von
  Eigenwerten und Eigenfunktionen. Math.- naturwiss. Diss.}\fg}, {G\"ottingen,
  30 S (1925).}, 1925.

\bibitem{Str}
{\scshape K.~Strebel}, \emph{Quadratic differentials}, Ergebnisse der
  Mathematik und ihrer Grenzgebiete (3) [Results in Mathematics and Related
  Areas (3)], vol.~5, Springer-Verlag, Berlin, 1984, xii+184~pages.

\bibitem{U}
{\scshape K.~Uhlenbeck}, {\og Generic properties of eigenfunctions\fg},
  \emph{Amer. J. Math.} \textbf{98} (1976), \cdrnumero 4, p.~1059-1078.

\bibitem{V}
{\scshape J.~Vilms}, {\og Totally geodesic maps\fg}, \emph{J. Differential
  Geometry} \textbf{4} (1970), p.~73-79.

\bibitem{Z17}
{\scshape S.~Zelditch}, {\og Logarithmic lower bound on the number of nodal
  domains\fg}, \emph{J. Spectr. Theory} \textbf{6} (2016), \cdrnumero 4,
  p.~1047-1086.

\end{thebibliography}

\end{document}